\documentclass[12pt,a4paper,leqno]{amsart} %optio titlepage!
\usepackage[T1]{fontenc}       % kirjaimet, joissa aksentteja (skandit)
\usepackage{amsfonts}          % AMS-fontit
\usepackage{amsmath}
\usepackage{amssymb}
\usepackage{overpic}
\usepackage{euscript} % \mathcal tuottaa hyvännäköisen fontin
\usepackage{eurosym}
\usepackage{comment}
\usepackage{mathrsfs} % calligraphic font via the command \mathscr{ABC}
\usepackage{ifthen}
\usepackage{esint} %k.a.integraali komennolla \fint
\usepackage{color}

\long\def\symbolfootnote[#1]#2{\begingroup%
\def\thefootnote{\fnsymbol{footnote}}\footnote[#1]{#2}\endgroup}

{\qed\vspace{5pt}}

\usepackage{amsthm}

\newtheoremstyle{lause}% name
{5pt}% space above
{5pt}% space below
{\slshape}% body font
{\parindent}% indent amount (empty = no indent)
{\bfseries}% theorem head font
{.}% punctuation after theorem head
{.5em}% space after theorem head
{}% theorem head spec (can be left empty, meaning 'normal')

\theoremstyle{lause}
\newtheoremstyle{maaritelma}% name
{5pt}% space above
{5pt}% space below
{\rmfamily}% body font
{\parindent}% indent amount (empty = no indent)
{\bfseries}% theorem head font
{.}% punctuation after theorem head
{.5em}% space after theorem head
{}% theorem head spec (can be left empty, meaning 'normal')

\theoremstyle{maaritelma}
\newtheoremstyle{lause}% name
{5pt}% space above
{5pt}% space below
{\slshape}% body font
{\parindent}% indent amount (empty = no indent)
{\bfseries}% theorem head font
{.}% punctuation after theorem head
{.5em}% space after theorem head
{}% theorem head spec (can be left empty, meaning 'normal')

\theoremstyle{lause}
\newtheorem{thm}{Theorem}[section]
\newtheorem{lem}[thm]{Lemma}

\newtheorem{cor}[thm]{Corollary}
\newtheorem{problem}[thm]{Problem}

\newtheoremstyle{maaritelma}% name
{5pt}% space above
{5pt}% space below
{\rmfamily}% body font
{\parindent}% indent amount (empty = no indent)
{\bfseries}% theorem head font
{.}% punctuation after theorem head
{.5em}% space after theorem head
{}% theorem head spec (can be left empty, meaning 'normal')

\theoremstyle{maaritelma}
\newtheorem{defn}[thm]{Definition}

\newtheorem{exmp}[thm]{Example}
\newtheorem{rem}[thm]{Remark}

\numberwithin{equation}{section}

\begin{document}

\thispagestyle{empty}

%\vspace*{-1.4truecm}\vbox{\noindent
%\footnotesize{Annales Academi\ae\ Scientiarum Fennic\ae\\
%Mathematica\\
%Volumen 43, 2018, 1--25}}

%\vspace{20pt}

\begin{center}

{\large{\textbf{Various concepts of Riesz energy of measures and
application to condensers with touching plates}}}

\vspace{18pt}

\textbf{Bent Fuglede and Natalia Zorii}

\vspace{18pt}

\footnotesize\textsl{Dedicated to Professor Stephen J.~Gardiner\\
on the occasion of his 60th birthday}\vspace{8pt}

\footnotesize{\address{Department of Mathematical Sciences,  Universitetsparken 5,  2100 Copenhagen, Denmark\\
fuglede@math.ku.dk}} \vspace{4pt}

\footnotesize{\address{Institute of Mathematics of National Academy
of Sciences of Ukraine, \linebreak Tereshchenkivska 3, 01601,
Kyiv-4, Ukraine\\
natalia.zorii@gmail.com }}

\end{center}

\vspace{12pt}

{\footnotesize{\textbf{Abstract.} We develop further the concept of
weak $\alpha$-Riesz energy with $\alpha\in(0,2]$ of Radon measures
$\mu$ on $\mathbb R^n$, $n\geqslant3$, introduced in our preceding
study and defined by $\int(\kappa_{\alpha/2}\mu)^2\,dm$, $m$
denoting the Lebesgue measure on $\mathbb R^n$. Here
$\kappa_{\alpha/2}\mu$ is the potential of $\mu$ relative to the
$\alpha/2$-Riesz kernel $|x-y|^{\alpha/2-n}$. This concept extends
that of standard $\alpha$-Riesz energy, and for $\mu$ with
$\kappa_{\alpha/2}\mu\in L^2(m)$ it coincides with that of
Deny--Schwartz energy defined with the aid of the Fourier transform.
We investigate minimum weak $\alpha$-Riesz energy problems with
external fields in both the unconstrained and constrained settings
for generalized condensers $(A_1,A_2)$ such that the closures of
$A_1$ and $A_2$ in $\mathbb R^n$ are allowed to intersect one
another. (Such problems with the standard $\alpha$-Riesz energy in
place of the weak one would be unsolvable, which justifies the need
for the concept of weak energy when dealing with condenser
problems.) We obtain sufficient and/or necessary conditions for the
existence of minimizers, provide descriptions of their supports and
potentials, and single out their characteristic properties. To this
end we have discovered an intimate relation between minimum weak
$\alpha$-Riesz energy problems over signed measures associated with
$(A_1,A_2)$ and minimum $\alpha$-Green energy problems over positive
measures carried by $A_1$. Crucial for our analysis of the latter
problems is the perfectness of the $\alpha$-Green kernel,
established in our recent paper. As an application of the results
obtained, we describe the support of the $\alpha$-Green equilibrium
measure.}}
%\symbolfootnote[0]{\quad
%https:/\!/doi.org/10.5186/aasfm.2016.41xx}
\symbolfootnote[0]{\quad 2010 Mathematics Subject Classification:
Primary 31C15.} \symbolfootnote[0]{\quad Key words: standard and
weak Riesz energy, Deny--Schwartz energy, minimum energy problems,
condensers with touching plates, external fields, constraints.}

\vspace{6pt}

\markboth{\textsl{Bent Fuglede and Natalia Zorii}} {\textsl{Minimum
energy problems for condensers with touching plates}}

\section{Introduction}

Throughout the paper we fix a natural $n\geqslant3$ and a real
$\alpha\in(0,2]$. Let $\mathfrak M(\mathbb R^n)$ stand for the
linear space of all real-valued Radon measures $\mu$ on $\mathbb
R^n$, equipped with the \textsl{vague\/} topology, i.e.\ the
topology of pointwise convergence on the class $C_0(\mathbb R^n)$ of
all (real-valued finite) continuous functions on $\mathbb R^n$ with
compact support. The \textsl{standard\/} concept of
\textsl{energy\/} of a (signed) Radon measure $\mu\in\mathfrak
M(\mathbb R^n)$ relative to the $\alpha$-\textsl{Riesz\/} kernel
$\kappa_\alpha(x,y):=|x-y|^{\alpha-n}$ on $\mathbb R^n$, $|x-y|$
being the Euclidean distance between $x,y\in\mathbb R^n$, is
introduced by
\begin{equation}\label{ren}E_\alpha(\mu):=E_{\kappa_\alpha}(\mu):=
\int\kappa_\alpha(x,y)\,d(\mu\otimes\mu)(x,y)\end{equation}
whenever $E_\alpha(\mu^+)+E_\alpha(\mu^-)$ or $E_\alpha(\mu^+,\mu^-)$ is finite, and finiteness of
$E_\alpha(\mu)$ means that $\kappa_\alpha$ is $\bigl(|\mu|\otimes|\mu|\bigr)$-integrable, i.e.\
$E_\alpha(|\mu|)<\infty$. Here $\mu^+$ and $\mu^-$ denote the positive and negative parts in the
Hahn--Jordan decomposition of a measure $\mu\in\mathfrak M(\mathbb R^n)$, $|\mu|:=\mu^++\mu^-$, and
\[E_\alpha(\mu^+,\mu^-):=E_{\kappa_\alpha}(\mu^+,\mu^-):=\int\kappa_\alpha(x,y)\,d(\mu^+\otimes\mu^-)(x,y)\]
is the (standard) $\alpha$-Riesz \textsl{mutual energy\/} of $\mu^+$
and $\mu^-$.

The $\alpha$-Riesz kernel is \textsl{strictly positive definite\/}
in the sense that $E_\alpha(\mu)$, $\mu\in\mathfrak M(\mathbb R^n)$,
is ${}\geqslant0$ (whenever defined), and equals $0$ only for
$\mu=0$. The set $\mathcal E_\alpha(\mathbb R^n)$ of all
$\mu\in\mathfrak M(\mathbb R^n)$ with $E_\alpha(\mu)<\infty$
therefore forms a pre-Hilbert space with the (standard) inner
product $\langle\mu,\nu\rangle_\alpha:=E_\alpha(\mu,\nu)$,
$\mu,\nu\in\mathcal E_\alpha(\mathbb R^n)$, and the (standard
energy) norm $\|\mu\|_\alpha:=\sqrt{E_\alpha(\mu)}$.

Fix an (open connected) domain $D$ in $\mathbb R^n$. An ordered pair
$\mathbf A=(A_1,A_2)$, where $A_1$ is a relatively closed subset of
$D$ and $A_2=D^c:=\mathbb R^n\setminus D$, is said to be a
(\textsl{generalized\/}) \textsl{condenser\/} in $\mathbb R^n$,
while $A_1$ and $A_2$ are termed the \textsl{positive\/} and
\textsl{negative plates\/} of $\mathbf A$. Note that although
$A_1\cap A_2=\varnothing$, $\mathrm{Cl}_{\mathbb R^n}A_1$ and $A_2$
may have points in common; otherwise we shall call $\mathbf
A=(A_1,A_2)$ a \textsl{standard\/} condenser. A measure
$\mu\in\mathfrak M(\mathbb R^n)$ is said to be \textsl{associated
with\/} a generalized condenser $\mathbf A$ if $\mu^+$ and  $\mu^-$
are carried by $A_1$ and $A_2$, respectively. The set $\mathfrak
M(\mathbf A)$ consisting of all those $\mu$ forms a convex cone in
$\mathfrak M(\mathbb R^n)$, and so does $\mathcal E_\alpha(\mathbf
A):= \mathfrak M(\mathbf A)\cap\mathcal E_\alpha(\mathbb R^n)$.

As a preparation for a study of minimum (standard) $\alpha$-Riesz
energy problems over subclasses of $\mathcal E_\alpha(\mathbf A)$,
$\mathbf A$ being a generalized condenser in $\mathbb R^n$, one
considered in a recent paper \cite{DFHSZ} the $\alpha$-Green kernel
$g=g_D^\alpha$ on $D$, associated with the $\alpha$-Riesz kernel
$\kappa_\alpha$ \cite[Chapter~IV, Section~5]{L}. It is claimed in
\cite[Lemma~2.4]{DFHSZ} that if a bounded positive Radon measure
$\nu$ on $D$ has finite $\alpha$-Green energy $E_g(\nu)$, defined by
(\ref{ren}) with $g$ in place of $\kappa_\alpha$, then the (signed)
Radon measure $\nu-\nu'$ on $\mathbb R^n$, $\nu'$ being the
$\alpha$-Riesz swept measure of $\nu$ onto $D^c$ \cite[Chapter~IV,
Section~5]{L}, must have finite (standard) $\alpha$-Riesz energy.
Regrettably, the short proof of Lemma~2.4 in \cite{DFHSZ} was
incomplete (a matter of $\infty-\infty$), and actually the lemma
fails in general, as seen by the counterexample given in
\cite[Appendix]{DFHSZ2}. To be precise, the quoted example shows
that there is a bounded positive Radon measure $\nu$ on $D$ with
finite $E_g(\nu)$ such that $E_\alpha(\nu-\nu')$ is not well
defined.

Below we argue that this failure is an indication that the standard notion of $\alpha$-Riesz energy of
signed measures is too restrictive when dealing with condenser problems.\footnote{Except for the case of a
standard condenser $\mathbf A$ in $\mathbb R^n$ with nonzero Euclidean distance between $A_1$ and $A_2$,
see Remarks~\ref{remark}, \ref{rem2} and Section~\ref{sec-5.3} below.}

The above mentioned error led to the fact that some of the
assertions announced in \cite{DFHSZ} had a gap in their proofs, and
actually they fail in general, as will be seen from
Example~\ref{counterex2}. Below we show that, nevertheless, these
assertions become valid (even in a stronger form) if we replace the
standard concept of $\alpha$-Riesz energy by a weaker concept (e.g.\
by that of Deny--Schwartz energy of measures treated as tempered
distributions).

As shown in \cite[Theorem~5.1]{FZ-Pot}, \cite[Lemma~2.4]{DFHSZ} quoted above does hold if we replace the
standard concept of $\alpha$-Riesz energy $E_\alpha(\mu)$ of a (signed) Radon measure $\mu$ on $\mathbb R^n$
by a weaker concept, denoted ${\dot E}_\alpha(\mu)$ and defined essentially
(see \cite[Definition~4.1]{FZ-Pot}) by
\[{\dot E}_\alpha(\mu)=\int(\kappa_{\alpha/2}\mu)^2\,dm,\]
$\kappa_{\alpha/2}\mu(\cdot):=\int\kappa_{\alpha/2}(\cdot,y)\,d\mu(y)$ being the pot\-en\-tial of $\mu$
relative to the $\alpha/2$-Riesz kernel. (Here and in the sequel $m$ denotes the Lebesgue measure on
$\mathbb R^n$.) Moreover, the concept of weak $\alpha$-Riesz energy, serving as a main tool of our analysis
in \cite{FZ-Pot}, enabled us to rectify partly the results announced in \cite{DFHSZ}
(see \cite[Section~6]{FZ-Pot}, see also Remark~\ref{rem-FZ-Pot} below for a short survey;
compare with a further development of \cite[Section~6]{FZ-Pot} achieved in the present study).

The set $\dot{\mathcal E}_\alpha(\mathbb R^n)$ of (signed) Radon measures $\mu$ on $\mathbb R^n$ with
finite ${\dot E}_\alpha(\mu)$, or equivalently with $\kappa_{\alpha/2}\mu\in L^2(m)$, forms a pre-Hilbert
space with the (weak) inner product $\langle\mu,\nu\rangle^\cdot_\alpha:=
\langle\kappa_{\alpha/2}\mu,\kappa_{\alpha/2}\nu\rangle_{L^2(m)}$,
$\mu,\nu\in\dot{\mathcal E}_\alpha(\mathbb R^n)$,
and the (weak energy) norm $\|\mu\|^{\cdot}_\alpha:=
\sqrt{{\dot E}_\alpha(\mu)}=\|\kappa_{\alpha/2}\mu\|_{L^2(m)}$ \cite[Section~4]{FZ-Pot}.
The Riesz composition identity \cite{R}
implies that
\[\mathcal E^+_\alpha(\mathbb R^n)=\dot{\mathcal E}^+_\alpha(\mathbb R^n)\text{ \ and \ }
\mathcal E_\alpha(\mathbb R^n)\subset\dot{\mathcal E}_\alpha(\mathbb R^n),\]
where $\dot{\mathcal E}^+_\alpha(\mathbb R^n):=\dot{\mathcal E}_\alpha(\mathbb R^n)\cap
\mathfrak M^+(\mathbb R^n)$; and moreover
\begin{equation}\label{eq}
E_\alpha(\nu)=\dot{E}_\alpha(\nu)\text{ \ for any\ }\nu\in\mathcal
E_\alpha(\mathbb R^n).\end{equation} However, as seen from
\cite[Appendix]{DFHSZ2} and \cite[Theorem~5.1]{FZ-Pot}, there exists
a (signed) measure $\mu\in\dot{\mathcal E}_\alpha(\mathbb R^n)$ such
that $E_\alpha(\mu)$ is not well defined, and hence
$\mu\notin\mathcal E_\alpha(\mathbb R^n)$. Thus $\mathcal
E_\alpha(\mathbb R^n)$ forms a \textsl{proper\/} subset of
$\dot{\mathcal E}_\alpha(\mathbb R^n)$, which by
\cite[Theorem~4.1]{FZ-Pot} is \textsl{dense\/} in $\dot{\mathcal
E}_\alpha(\mathbb R^n)$ in the topology determined by the weak
energy norm as well as in the (induced) vague topology.

Let $S_\alpha^*$ denote the Hilbert space of all real-valued
tempered distributions $T\in S^*$ on $\mathbb R^n$ (see \cite{S})
with finite Deny--Schwartz energy
\[\|T\|_{S_\alpha^*}^2:=\int_{\mathbb
R^n}\frac{\left|\mathcal F[T](y)\right|^2}{|y|^\alpha}\,dm(y),\]
where $\mathcal F[T]$ is the Fourier transform of $T\in S^*$. We refer to \cite{D1,D2}
(see also \cite[Chapter~VI, Section~1]{L}) for the definition of $\|T\|_{S_\alpha^*}^2$ as well as
the properties of the space $S_\alpha^*$.
As shown in \cite[Theorem~4.2]{FZ-Pot}, every $\nu\in\dot{\mathcal E}_\alpha(\mathbb R^n)$ can be treated as
an element of $S_\alpha^*$ with
\[\|\nu\|_{S_\alpha^*}^2=C_{n,\alpha}\|\nu\|^{\cdot\,2}_\alpha,\]
$C_{n,\alpha}\in(0,\infty)$ depending on $n$ and $\alpha$ only, and
moreover $S_\alpha^*$ is a completion of the pre-Hil\-bert space
$\dot{\mathcal E}_\alpha(\mathbb R^n)$ in the topology determined by
the Deny--Schwartz norm $\|\cdot\|_{S_\alpha^*}$. (This result for
$\mathcal E_\alpha(\mathbb R^n)$ in place of $\dot{\mathcal
E}_\alpha(\mathbb R^n)$ goes back to Deny \cite{D1}.)

Thus the concept of Deny--Schwartz energy coincides (up to a
constant factor) with that of weak $\alpha$-Riesz energy if
restricted to measures of the class $\dot{\mathcal E}_\alpha(\mathbb
R^n)$. It is however still unknown whether these two concepts are
identical if considered over \textsl{all\/} (signed) Radon measures
on $\mathbb R^n$ (see an open question raised in \cite[p.~85]{D2}
and \cite[Section~7]{FZ-Pot}, the former work dealing with positive
measures only).

In the present paper we proceed further with a study of minimum weak
$\alpha$-Riesz, or equivalently Deny--Schwartz, energy problems for
a generalized condenser $\mathbf A$, initiated in \cite{FZ-Pot}.
Similarly as in \cite{DFHSZ}, the measures are now influenced
additionally by external fields $f$ and/or constraints $\sigma$ (see
Section~\ref{sec-statement} for the precise formulations of the
problems in question).

It is shown in Theorem~\ref{th-equiv} below that \cite[Lemma~4.2]{DFHSZ} on the equivalence of
\cite[Problem~3.1]{DFHSZ} and \cite[Problem~3.2]{DFHSZ} becomes valid if we require additionally that
the plates $A_1$ and $A_2=D^c$ of the condenser $\mathbf A$ satisfy the separation
condition\footnote{Because of the error in \cite[Lemma~2.4]{DFHSZ}, \cite[Lemma~4.2]{DFHSZ} was unjustified
in the stated form. Moreover, it fails in general, as will be seen from Example~\ref{counterex2}.}
\begin{equation}\label{dist}{\rm dist}\,(A_1,D^c):=\inf_{(x,y)\in A_1\times D^c}\,|x-y|>0.\end{equation}
(Note that such $\mathbf A$ is certainly standard.)
As seen from the quoted theorem, the class of measures admissible in \cite[Problem~3.1]{DFHSZ} can then
equivalently be defined as the set $\widetilde{\mathcal H}$ of $\mu-\mu'$ where $\mu$ ranges over certain
positive measures carried by $A_1$ and having finite $\alpha$-Green energy, while $\mu'$ is the
$\alpha$-Riesz balayage of $\mu$ onto $D^c$. Omitting now the separation condition (\ref{dist}), we see
from Theorem~\ref{lem-equ} below that the stated equivalence of \cite[Problem~3.1]{DFHSZ} and
\cite[Problem~3.2]{DFHSZ}, the former now being applied to $\widetilde{\mathcal H}$ serving as a (new)
class of admissible measures, remains valid if the standard $\alpha$-Riesz energy is replaced by the
weak energy.

Sufficient and/or necessary conditions for the solvability of the minimum weak $\alpha$-Riesz energy
problems are provided in Theorems~\ref{main-th}, \ref{infcap} and Corollary~\ref{cor-n-s}. See also
Theorems~\ref{desc-th1}, \ref{cor-desc}, \ref{desc-infty}, and \ref{cor-infty-2} where such criteria are
formulated in terms of variational inequalities for the $\alpha$-Riesz potentials. An analysis of the
supports of the minimizers is given by Theorem~\ref{desc-sup}, \ref{cor-desc}, and \ref{cor-infty-3}.
The results obtained are in fact stronger than those announced in \cite{DFHSZ} since a crucial key to
our current proofs is the perfectness of the $\alpha$-Green kernel, established in our recent study
\cite{FZ-Fin}. As an application of the quoted results, we describe the support of the $\alpha$-Green
equilibrium measure (see Theorem~\ref{eq-m-desc}).

Example~\ref{counterex2} shows that the above mentioned assertions in general fail if the weak $\alpha$-Riesz
energy is replaced by the standard energy. This justifies the need for the concept of weak
$\alpha$-Riesz energy when dealing with condenser problems.

\begin{rem}The generalized condensers $\mathbf A=(A_1,D^c)$ such that the \textsl{unconstrained\/}
minimum weak $\alpha$-Riesz energy problems are solvable differ
drastically from those for which the solvability occurs in the
\textsl{constrained\/} setting. Indeed, if $f=0$ and $\sigma=\infty$
(\textsl{no\/} external field and \textsl{no\/} constraint), then
the solvability of the problems implies that $c_\alpha(\partial
D\cap\mathrm{Cl}_{\mathbb R^n}A_1)=0$, $c_\alpha(\cdot)$ being the
$\alpha$-Riesz capacity.\footnote{Throughout the paper $\partial Q$
denotes the boundary of a set $Q\subset\mathbb R^n$ relative to
$\mathbb R^n$.} But if the constraint in question is
\textsl{bounded\/}, then the problems turn out to be solvable even
if $A_1=D$; see Remark~\ref{dr-dif} below for details.
\end{rem}

\begin{rem}The results announced in \cite{DFHSZ} have also been rectified in part in a recent work
\cite{DFHSZ2}. However, minimum (standard) $\alpha$-Riesz energy
problems for a generalized condenser $\mathbf A$ were analyzed in
\cite{DFHSZ2} only in the constrained setting, and the constraints
were required to be of finite (standard) $\alpha$-Riesz energy,
which was crucial for the proofs in \cite{DFHSZ2}. The concept of
weak $\alpha$-Riesz energy, serving as a main tool of our present
study, enables us to rectify the results announced in \cite{DFHSZ}
in both the unconstrained and constrained settings, and the
constraints in question are no longer required to satisfy any
additional assumptions.\end{rem}

\section{Preliminaries}\label{sec:princ}

Let $X$ be a locally compact (Hausdorff) space \cite[Chapter~I,
Section~9, n$^\circ$\,7]{B1}, to be specified below. For the goals
of the present study it is enough to assume that $X$ is metrizable
and \textsl{countable at infinity\/}, where the latter means that
$X$ can be represented as a countable union of compact sets
\cite[Chapter~I, Section~9, n$^\circ$\,9]{B1}. Then the vague
topology\footnote{We shall tacitly assume to be known this and other
notions defined for $X=\mathbb R^n$ in the Introduction.} on
$\mathfrak M(X)$ satisfies the first axiom of countability
\cite[Remark~2.5]{DFHSZ1}, and vague convergence is entirely
determined by convergence of sequences. The vague topology on
$\mathfrak M(X)$ is Hausdorff, and hence a vague limit of any
sequence in $\mathfrak M(X)$ is \textsl{unique\/} (whenever it
exists). We denote by $S^\mu_{X}=S(\mu)$ the \textsl{support\/} of
$\mu\in\mathfrak M(X)$. A measure $\mu$ is said to be
\textsl{bounded\/} if $|\mu|(X)<\infty$. Let $\mathfrak M^+(X)$
stand for the (convex, vaguely closed) cone of all positive
$\mu\in\mathfrak M(X)$.

Given a set $Q\subset X$, let $\mathfrak M^+(Q;X)$ consist of all
$\mu\in\mathfrak M^+(X)$ \textsl{carried by\/} $Q$, which means that
$X\setminus Q$ is locally $\mu$-negligible, or equivalently that $Q$
is $\mu$-meas\-ur\-able and $\mu=\mu|_Q$, where $\mu|_Q=1_Q\cdot\mu$
is the trace (restriction) of $\mu$ on $Q$ \cite[Chapter~V,
Section~5, n$^\circ$\,3, Example]{B2}. (Here $1_Q$ denotes the
indicator function of $Q$.) If $Q$ is closed, then $\mu$ is carried
by $Q$ if and only if it is supported by $Q$, i.e.\ $S(\mu)\subset
Q$. It follows from the countability of $X$ at infinity that the
concept of local $\mu$-neg\-lig\-ibility coincides with that of
$\mu$-negligibility; and hence $\mu\in\mathfrak M^+(Q;X)$ if and
only if $\mu^*(X\setminus Q)=0$, $\mu^*(\cdot)$ being the
\textsl{outer measure\/} of a set. Denoting by $\mu_*(\cdot)$ the
\textsl{inner measure\/} of a set, for any $\mu\in\mathfrak
M^+(Q;X)$ we thus get
\[\mu^*(Q)=\mu_*(Q)=:\mu(Q).\]
Write $\mathfrak M^+(Q,q;X):=\bigl\{\mu\in\mathfrak M^+(Q;X):\ \mu(Q)=q\bigr\}$, where $q>0$.

The following well known fact (see e.g.\ \cite[Section~1.1]{F1}) will often be used.

\begin{lem}\label{lemma-semi}Let\/ $\psi$ be a lower semicontinuous\/ {\rm(}l.s.c.\/{\rm)} function
on\/ $X$, nonnegative unless\/ $X$ is compact. The mapping\/
$\mu\mapsto\langle\psi,\mu\rangle:=\int\psi\,d\mu$ is then vaguely
l.s.c.\ on\/ $\mathfrak M^+(X)$.\end{lem}

A (function) \textsl{kernel\/} on $X$ is defined as a symmetric
l.s.c.\ function $\kappa:X\times X\to[0,\infty]$. Given
$\mu,\nu\in\mathfrak M(X)$, we denote by $E_\kappa(\mu,\nu)$ and
$\kappa\mu$ the (standard) \textsl{mutual energy\/} and the
\textsl{potential\/} relative to the kernel $\kappa$. Let $\mathcal
E_\kappa(X)$ consist of all $\mu\in\mathfrak M(X)$ whose (standard)
energy $E_\kappa(\mu)$ is finite, which means that
$E_\kappa(|\mu|)<\infty$, and let $\mathcal E^+_\kappa(X):=\mathcal
E_\kappa(X)\cap\mathfrak M^+(X)$.

\textsl{In all that follows we assume a kernel $\kappa$ to be
strictly positive definite\/}. Then $\mathcal E_\kappa(X)$ forms a
pre-Hil\-bert space with the (standard) inner product
$\langle\mu,\nu\rangle_\kappa:=E_\kappa(\mu,\nu)$,
$\mu,\nu\in\mathcal E_\kappa(X)$, and the (standard energy) norm
$\|\mu\|_\kappa:=\sqrt{E_\kappa(\mu)}$ (see \cite{F1}). The
(Hausdorff) topology on $\mathcal E_\kappa(X)$ determined by the
(standard energy) norm $\|\cdot\|_\kappa$ is termed
\textsl{strong\/}.

For a set $Q\subset X$ write $\mathcal E_\kappa^+(Q;X):=\mathcal
E_\kappa(X)\cap\mathfrak M^+(Q;X)$ and $\mathcal
E_\kappa^+(Q,q;X):=\mathcal E_\kappa(X)\cap\mathfrak M^+(Q,q;X)$,
where $q\in(0,\infty)$. The (\textsl{inner\/}) \textsl{capacity\/}
$c_\kappa(Q)$ of $Q$ relative to the kernel $\kappa$ is defined by
\begin{equation}\label{cap-def}c_\kappa(Q)^{-1}:=\inf_{\mu\in\mathcal
E_\kappa^+(Q,1;X)}\,\|\mu\|_\kappa^2\end{equation} (see e.g.\
\cite{F1,O}). Then $0\leqslant c_\kappa(Q)\leqslant\infty$. (Here
and in the sequel the infimum over the empty set is taken to be
$+\infty$. We also set $1\bigl/(+\infty)=0$ and $1\bigl/0=+\infty$.)

Because of the strict positive definiteness of the kernel $\kappa$,
$c_\kappa(K)<\infty$ for every compact set $K\subset X$.
Furthermore, by \cite[p.~153, Eq.~(2)]{F1},
\begin{equation}\label{compact}c_\kappa(Q)=\sup\,c_\kappa(K)\quad(K\subset Q, \ K\text{\ compact}).
\end{equation}

An assertion $\mathcal U(x)$ involving a variable point $x\in X$ is
said to hold \textsl{$c_\kappa$-ne\-ar\-ly everywhere\/}
(\textsl{$c_\kappa$-n.e.}) on $Q\subset X$ if $c_\kappa(N)=0$, where
$N$ consists of all $x\in Q$ for which $\mathcal U(x)$ fails. It is
often used that $c_\kappa(N)=0$ if and only if $\mu_*(N)=0$ for
every $\mu\in\mathcal E_\kappa^+(X)$ \cite[Lemma~2.3.1]{F1}. We
shall sometimes need also the concept of \textsl{$c_\kappa$-qua\-si
everywhere\/} (\textsl{$c_\kappa$-q.e.}) where the exceptional set
$N$ is supposed to have \textsl{outer\/} capacity zero. These two
concepts of negligibility coincide if the exceptional sets are
\textsl{capacitable\/} relative to the kernel~$\kappa$.

As in \cite[p.\ 134]{L}, we call a measure $\mu\in\mathfrak M(X)$
\textsl{$c_\kappa$-absolutely continuous\/} if $\mu(K)=0$ for every
compact set $K\subset X$ with $c_\kappa(K)=0$. It follows from
(\ref{compact}) that for such $\mu$, $|\mu|_*(Q)=0$ for every
$Q\subset X$ with $c_\kappa(Q)=0$. Hence, every $\mu\in\mathcal
E_\kappa(X)$ is $c_\kappa$-ab\-sol\-utely continuous; but not
conversely \cite[pp.~134--135]{L}.

\begin{defn}\label{def-perf}Following~\cite{F1}, we call a (strictly positive definite)
kernel $\kappa$ \textsl{perfect\/} if every strong Cauchy sequence
in $\mathcal E_\kappa^+(X)$ converges strongly to any of its vague
cluster points\footnote{It follows from Theorem~\ref{fu-complete}
that for a perfect kernel such a vague cluster point exists and is
unique.}.\end{defn}

\begin{rem}\label{ex-perf} On $X=\mathbb R^n$, $n\geqslant3$, the $\alpha$-Riesz kernel
$\kappa_\alpha(x,y)=|x-y|^{\alpha-n}$, $\alpha\in(0,n)$, is strictly
positive definite and moreover perfect \cite{D1,D2}; thus so is the
Newtonian kernel $\kappa_2(x,y)=|x-y|^{2-n}$ \cite{Ca}. Recently it
has been shown by the present authors that if $X$ is an open set $D$
in $\mathbb R^n$, $n\geqslant3$, and $g^\alpha_D$, $\alpha\in(0,2]$,
is the $\alpha$-Green kernel on $D$ \cite[Chapter~IV, Section~5]{L},
then $\kappa=g^\alpha_D$ likewise is strictly positive definite and
moreover perfect \cite[Theorems~4.9, 4.11]{FZ-Fin}.\end{rem}

\begin{thm} [{\rm see \cite{F1}}]\label{fu-complete} If a kernel\/ $\kappa$ is perfect, then
the cone\/ $\mathcal E_\kappa^+(X)$ is strongly complete and the
strong topology on\/ $\mathcal E_\kappa^+(X)$ is finer than the\/
{\rm(}induced\/{\rm)} vague topology on\/ $\mathcal
E_\kappa^+(X)$.\end{thm}

\begin{rem}In contrast to Theorem~\ref{fu-complete}, for a perfect kernel $\kappa$ the whole pre-Hilbert
space $\mathcal E_\kappa(X)$ is in general strongly
\textsl{incomplete\/}, and this is the case even for the
$\alpha$-Riesz kernel of order $\alpha\in(1,n)$ on $\mathbb R^n$,
$n\geqslant 3$ (see \cite{Ca}); compare with the following
Remark~\ref{remark}.\end{rem}

\begin{rem}\label{remark} The concept of perfect kernel is an efficient tool in minimum energy problems
over classes of \textsl{positive\/} Radon measures with finite
energy. Indeed, if $Q\subset X$ is closed and $\kappa$ is perfect,
then problem (\ref{cap-def}) has a (unique) solution
$\mu_{Q,\kappa}$ if and only if $0<c_\kappa(Q)<\infty$
\cite[Theorem~4.1]{F1}; such $\mu_{Q,\kappa}$ is termed the
(\textsl{inner\/}) \textsl{$\kappa$-cap\-ac\-itary measure\/} on
$Q$. Later the concept of perfectness has been shown to be efficient
also in minimum (standard) energy problems over classes of
(\textsl{signed\/}) measures associated with a \textsl{standard
condenser\/} (see \cite{ZPot1}--\cite{ZPot2}; see also the earlier
study \cite{ZR} pertaining to the $\alpha$-Riesz kernel on $\mathbb
R^n$). The approach developed in \cite{ZPot1}--\cite{ZPot2}
substantially used \textsl{the assumption of the boundedness of the
kernel on the product of the oppositely charged plates of a
condenser\/},\footnote{For any classical kernel $\kappa$ on $\mathbb
R^n$ the quoted assumption of the boundedness of $\kappa$ on the
product of the oppositely charged plates is equivalent to the
separation condition~(\ref{dist}).} which made it possible to extend
Cartan's proof \cite{Ca} of the strong completeness of the cone
$\mathcal E_{\kappa_2}^+(\mathbb R^n)$ of all \textsl{positive\/}
measures on $\mathbb R^n$ with finite Newtonian energy to an
arbitrary perfect kernel $\kappa$ on a locally compact space $X$ and
suitable classes of (\textsl{signed\/}) measures $\mu\in\mathcal
E_\kappa(X)$. In turn, this strong completeness theorem for metric
subspaces of signed $\mu\in\mathcal E_\kappa(X)$ made it possible to
develop a fairly general theory of standard condensers, actually
even with \textsl{countably many\/} plates.
\end{rem}

A set $Q\subset X$ is said to be \textsl{locally closed\/} in $X$ if
for every $x\in Q$ there is a neighborhood $V$ of $x$ in $X$ such
that $V\cap Q$ is a closed subset of the subspace $Q$
\cite[Chapter~I, Section~3, Definition~2]{B1}, or equivalently if
$Q$ is the intersection of an open and a closed subset of $X$
\cite[Chapter~I, Section~3, Proposition~5]{B1}. The latter implies
that this $Q$ is universally measurable, and hence $\mathfrak
M^+(Q;X)$ consists of all the restrictions $\mu|_Q$ where $\mu$
ranges over $\mathfrak M^+(X)$. On the other hand, by
\cite[Chapter~I, Section~9, Proposition~13]{B1} a locally closed set
$Q$ itself can be thought of as a locally compact subspace of $X$.
Thus $\mathfrak M^+(Q;X)$ consists, in fact, of all those
$\nu\in\mathfrak M^+(Q)$ for each of which there is
$\hat{\nu}\in\mathfrak M^+(X)$ with the property
\begin{equation}\label{extend}\hat{\nu}(\varphi)=\int\varphi|_{Q}\,d\nu\text{ \ for every\ }
\varphi\in C_0(X).\end{equation} We say that such $\hat{\nu}$
\textsl{extends\/} $\nu\in\mathfrak M^+(Q)$ by $0$ off $Q$ to all of
$X$. A sufficient condition for (\ref{extend}) to hold is that $\nu$
be bounded.

\section{$\alpha$-Riesz balayage and $\alpha$-Green kernel}\label{sec-bala} In the rest of the paper
fix $n\geqslant3$, $\alpha\in(0,2]$ and a domain $D\subset\mathbb
R^n$ with $c_{\kappa_\alpha}(D^c)>0$, and assume that either
$\kappa=\kappa_\alpha$ is the \textsl{$\alpha$-Riesz kernel\/} on
$X=\mathbb R^n$, or $\kappa=g_D^\alpha$ is the
\textsl{$\alpha$-Green kernel\/} on $X=D$. We simply write $\alpha$
instead of $\kappa_\alpha$ if $\kappa_\alpha$ serves as an index,
and we use the short form `n.e.' instead of `$c_\alpha$-n.e.' if
this will not cause any mis\-under\-standing.

When speaking of a positive Radon measure $\mu$ on $\mathbb R^n$, we always tacitly assume that for
the given $\alpha$, $\kappa_\alpha\mu$ is not identically infinite. This implies that
\begin{equation}\label{1.3.10}\int_{|y|>1}\,\frac{d\mu(y)}{|y|^{n-\alpha}}<\infty\end{equation}
(see \cite[Eq.~(1.3.10)]{L}), and consequently that
$\kappa_\alpha\mu$ is finite ($c_\alpha$-)n.e.\ on $\mathbb R^n$
\cite[Chap\-ter~III, Section~1]{L}; these two implications can
actually be reversed.

\begin{defn} A (signed) measure $\nu\in\mathfrak M(D)$ is termed \textsl{extendible\/} if there are
$\widehat{\nu^+}$ and $\widehat{\nu^-}$ extending $\nu^+$ and
$\nu^-$, respectively, by $0$ off $D$ to all of $\mathbb R^n$, see
(\ref{extend}), and if these $\widehat{\nu^+}$ and $\widehat{\nu^-}$
satisfy the general convention (\ref{1.3.10}). We identify this
$\nu\in\mathfrak M(D)$ with its extension
$\hat\nu:=\widehat{\nu^+}-\widehat{\nu^-}$, and we therefore write
$\hat\nu=\nu$.\end{defn}

Every bounded measure $\nu\in\mathfrak M(D)$ is extendible. The converse holds if $D$ is bounded,
but not in general (e.g.\ not if $D^c$ is compact). The set of all extendible measures consists of all
the restrictions $\mu|_D$ where $\mu$ ranges over $\mathfrak M(\mathbb R^n)$, see the end of
Section~\ref{sec:princ}. Also note that for any extendible measure $\nu\in\mathfrak M(D)$,
$\kappa_{\alpha}\nu$ is finite n.e.\ on $\mathbb R^n$, for $\kappa_\alpha\nu^\pm$ is so.

The \textsl{$\alpha$-Green kernel\/} $g=g_D^\alpha$ on $D$ is
defined by
\[g^\alpha_D(x,y):=\kappa_\alpha\varepsilon_y(x)-
\kappa_\alpha\varepsilon_y^{D^c}(x)\text{ \ for all\ }x,y\in D,\]
where $\varepsilon_y$ denotes the unit Dirac measure at a point $y$
and $\varepsilon_y^{D^c}$ its \textsl{$\alpha$-Riesz balayage\/}
(sweeping) onto the (closed) set $D^c$, determined uniquely in the
frame of the classical approach by \cite[Theorem~3.6]{FZ-Fin}
pertaining to positive Radon measures on $\mathbb R^n$. See also the
book by Bliedtner and Hansen \cite{BH} where balayage is studied in
the setting of balayage spaces.

We shall simply write $\mu'$ instead of $\mu^{D^c}$ when speaking of
the $\alpha$-Riesz balayage of $\mu\in\mathfrak M^+(D;\mathbb R^n)$
onto $D^c$. According to \cite[Corollaries~3.19, 3.20]{FZ-Fin}, for
any $\mu\in\mathfrak M^+(D;\mathbb R^n)$ \textsl{the balayage\/
$\mu'$ is\/ $c_\alpha$-ab\-sol\-ut\-ely continuous and it is
determined uniquely by the relation
\begin{equation}\label{bal-eq}\kappa_\alpha\mu'=\kappa_\alpha\mu\text{ \ n.e.\ on\ }D^c\end{equation}
among the\/ $c_\alpha$-absolutely continuous measures supported
by\/} $D^c$. Furthermore, there holds the integral representation
\cite[Theorem~3.17]{FZ-Fin}\footnote{In the literature the integral
representation (\ref{int-repr}) seems to have been more or less
taken for granted, though it has been pointed out in
\cite[Chapter~V, Section~3, n$^\circ$\,1]{B2} that it requires that
the family $(\varepsilon_y')_{y\in D}$ be
\textsl{$\mu$-ad\-equ\-ate\/} in the sense of \cite[Chapter~V,
Section~3, Definition~1]{B2}; see also counterexamples (without
$\mu$-ad\-equ\-acy) in Exercises~1 and~2 at the end of that section.
A proof of this adequacy has therefore been given in
\cite[Lemma~3.16]{FZ-Fin}.}
\begin{equation}\label{int-repr}\mu'=\int\varepsilon_y'\,d\mu(y).\end{equation}
If moreover\/ $\mu\in\mathcal E_\alpha^+(D;\mathbb R^n)$, then the
balayage $\mu'$ is in fact \textsl{the orthogonal projection of\/
$\mu$ onto the convex cone\/ $\mathcal E^+_\alpha(D^c;\mathbb R^n)$}
\cite[Theorem~3.1]{FZ-Fin}, i.e.\ $\mu'\in\mathcal
E^+_\alpha(D^c;\mathbb R^n)$ and
\begin{equation}\label{proj}\|\mu-\theta\|_\alpha>\|\mu-\mu'\|_\alpha\text{ \ for all \ }
\theta\in\mathcal E^+_\alpha(D^c;\mathbb R^n), \ \theta\ne\mu'.\end{equation}

If now $\nu\in\mathfrak M(D)$ is an extendible (signed) measure,
then $\nu':=\nu^{D^c}:=(\nu^+)'-(\nu^-)'$ is said to be a
\textsl{balayage\/} of $\nu$ onto $D^c$. It follows from
\cite[p.~178, Remark]{L} that the balayage $\nu'$ is determined
uniquely by (\ref{bal-eq}) with $\nu$ in place of $\mu$ among the
$c_\alpha$-ab\-sol\-ut\-ely continuous (signed) measures supported
by~$D^c$.

\begin{defn}[{\rm see \cite[Theorem~VII.13]{Brelo2}}] A closed set $Q\subset\mathbb R^n$ is said to be
\textsl{$\alpha$-thin at infinity\/} if either $Q$ is compact, or
the inverse of $Q$ relative to $S(0,1):=\bigl\{x\in\mathbb R^n:
|x|=1\bigr\}$ has $x=0$ as an $\alpha$-irregular boundary point
(cf.\ \cite[Theorem~5.10]{L}).\end{defn}

\begin{rem}\label{rem-thin}Any closed set $Q$ that is not $\alpha$-thin at infinity is of infinite
capacity $c_\alpha(Q)$. Indeed, by the Wiener criterion of $\alpha$-regularity, $Q$ is not $\alpha$-thin
at infinity if and only if
\[\sum_{k\in\mathbb N}\,\frac{c_\alpha(Q_k)}{q^{k(n-\alpha)}}=\infty,\]
where $q>1$ and $Q_k:=Q\cap\bigl\{x\in\mathbb R^n: q^k\leqslant|x|<q^{k+1}\bigr\}$, while by
\cite[Lemma~5.5]{L} $c_\alpha(Q)<\infty$ is equivalent to the relation
\[\sum_{k\in\mathbb N}\,c_\alpha(Q_k)<\infty.\]
These observations also imply that the converse is not true, i.e.\
there is $Q$ with $c_\alpha(Q)=\infty$, but $\alpha$-thin at
infinity (see also \cite[pp.~276--277]{Ca2}).\end{rem}

\begin{exmp}[{\rm see \cite[Example~5.3]{ZPot1}}] Let $n=3$ and $\alpha=2$. Define the rotation body
\begin{equation}\label{descr}Q_\varrho:=\bigl\{x\in\mathbb R^3: \
0\leqslant x_1<\infty, \ x_2^2+x_3^2\leqslant\varrho^2(x_1)\bigr\},\end{equation}
where $\varrho$ is given by one of the following three formulae:
\begin{align}
\label{c1}\varrho(x_1)&=x_1^{-s}\text{ \ with \ }s\in[0,\infty),\\
\label{c2}\varrho(x_1)&=\exp(-x_1^s)\text{ \ with \ }s\in(0,1],\\
\label{c3}\varrho(x_1)&=\exp(-x_1^s)\text{ \ with \ }s\in(1,\infty).
\end{align}
Then $Q_\varrho$ is not $2$-thin at infinity if $\varrho$ is defined
by (\ref{c1}), $Q_\varrho$ is $2$-thin at infinity but has infinite
Newtonian capacity if $\varrho$ is given  by (\ref{c2}), and finally
$c_2(Q_\varrho)<\infty$ if (\ref{c3}) holds.
\end{exmp}

\begin{thm} [{\rm see \cite[Theorem~3.22]{FZ-Fin}}]\label{bal-mass-th} The set\/ $D^c$ is not\/
$\alpha$-thin at infinity if and only if for every bounded measure\/
$\mu\in\mathfrak M^+(D)$ we have $\mu'(\mathbb R^n)=\mu(\mathbb
R^n)$.\footnote{In general, $\nu^{D^c}(\mathbb
R^n)\leqslant\nu(\mathbb R^n)$ for every $\nu\in\mathfrak
M^+(\mathbb R^n)$ \cite[Theorem~3.11]{FZ-Fin}.}
\end{thm}

\begin{thm}[{\rm see \cite[Theorem~4.12]{FZ-Fin}}]\label{th-equi}
For any relatively closed subset\/ $F$ of\/ $D$ with\/
$c_g(F)<\infty$ there exists a unique\/ $\alpha$-Green equilibrium
measure on\/ $F$, i.e.\ a measure\/
$\gamma_{F,g}=\gamma_F\in\mathcal E^+_g(F;D)$ such that\/
$\gamma_F(D)=\|\gamma_F\|_g^2=c_g(F)$ and\/\footnote{If $N$ is a
given subset of $D$, then $c_g(N)=0$ if and only if $c_\alpha(N)=0$
\cite[Lemma~2.6]{DFHSZ}. Thus any assertion involving a variable
point holds ($c_\alpha$-)n.e.\ on $Q\subset D$ if and only if it
holds $c_g$-n.e.\ on $Q$.\label{foot-ga}}
\begin{align}
\label{geq-1'}g\gamma_F&=1\text{ \ {\rm($c_\alpha$-)}n.e.\ on\ }F,\\
g\gamma_F&\leqslant1\text{ \ on\ }D.\notag
\end{align}
This\/ $\gamma_F$ is characterized uniquely within\/ $\mathcal
E^+_g(F;D)$ by\/ {\rm(\ref{geq-1'})}, and it is the\/
{\rm(}unique\/{\rm)} solution to the problem of minimizing\/
$E_g(\nu)$ over the {\rm(}convex\/{\rm)} class\/ $\Gamma_F$ of all\/
{\rm(}signed\/{\rm)} $\nu\in\mathcal E_g(D)$ with the property\/
$g\nu\geqslant1$ n.e.\ on\/ $F$. That is,
\begin{equation}\label{alt}c_g(F)=\|\gamma_F\|_g^2=\min_{\nu\in\Gamma_F}\,\|\nu\|_g^2.\end{equation}
If\/ $I_{F,\alpha}$ consists of all\/ $\alpha$-irregular points
of\/~$F$, then\/ {\rm(\ref{geq-1'})} can be specified as follows:
\begin{equation}\label{geq-2}g\gamma_F=1\text{ \ on \ }F\setminus I_{F,\alpha}.\end{equation}
\end{thm}

\begin{rem}\label{rem-eq-capacit}If\/ $F$ is a relatively closed subset of\/ $D$ with\/ $0<c_g(F)<\infty$,
then
\[\gamma_{F,g}=c_g(F)\mu_{F,g},\]
where $\mu_{F,g}$ is the (unique) $g$-capacitary measure on $F$
(which exists, see Remarks~\ref{ex-perf} and~\ref{remark}).
\end{rem}

\begin{cor}\label{lusin}If\/ $F$ is a relatively closed subset of\/ $D$ with\/ $c_g(F)<\infty$,
then\/ \[c_\alpha\bigl(\partial D\cap\mathrm{Cl}_{\mathbb
R^n}F\bigr)=0.\]
\end{cor}

\begin{proof} According to Theorem~\ref{th-equi}, there is the $\alpha$-Green equilibrium measure
$\gamma=\gamma_F$ on $F$. By Lusin's type theorem
\cite[Theorem~3.6]{L} applied to each of $\kappa_\alpha\gamma$ and
$\kappa_\alpha\gamma'$, there exists for any $\varepsilon>0$ an open
set $\Omega\subset\mathbb R^n$ with $c_\alpha(\Omega)<\varepsilon$
such that $\kappa_\alpha\gamma$ and $\kappa_\alpha\gamma'$ are both
continuous relative to $\mathbb R^n\setminus\Omega$. Since there is
no loss of generality in assuming $I_{F,\alpha}\subset\Omega$, we
thus get from Lemma~\ref{l-hatg} below and (\ref{geq-2})
\[\kappa_\alpha\gamma=\kappa_\alpha\gamma'+1\text{ \ on \ }(\mathrm{Cl}_{\mathbb R^n}F)\setminus\Omega.\]
As $\varepsilon$ is arbitrary,
$\kappa_\alpha\gamma=\kappa_\alpha\gamma'+1$ $c_\alpha$-q.e.\ on
$\partial D\cap\mathrm{Cl}_{\mathbb R^n}F$. But
$\kappa_\alpha\gamma=\kappa_\alpha\gamma'$ holds $c_\alpha$-n.e.\ on
$D^c$ by (\ref{bal-eq}), hence $c_\alpha$-q.e.\ because
$\bigl\{\kappa_\alpha\gamma\ne\kappa_\alpha\gamma'\bigr\}$ is a
Borel set, and the corollary follows.\end{proof}

The following three known assertions establish relations between potentials and standard energies
relative to the kernels $\kappa_\alpha$ and $g=g^\alpha_D$.

\begin{lem}[{\rm see \cite[Lemma~3.4]{DFHSZ2}}]\label{l-hatg}  For any extendible\/ {\rm(}signed\/{\rm)}
measure\/ $\mu\in\mathfrak M(D)$ the $\alpha$-Green potential\/
$g\mu$ is finite\/ {\rm($c_\alpha$-)}n.e.\ on\/ $D$ and given by\/
$g\mu=\kappa_\alpha\mu-\kappa_\alpha\mu'$.
\end{lem}

\begin{lem} [{\rm see \cite[Lemma~3.5]{DFHSZ2}}]\label{l-hen'} Suppose that\/ $\mu\in\mathfrak M(D)$
is extendible and the extension belongs to\/ $\mathcal
E_\alpha(\mathbb R^n)$. Then\/ $\mu\in\mathcal E_g(D)$,
$\mu-\mu'\in\mathcal E_\alpha(\mathbb R^n)$ and moreover
\begin{equation}\label{gr}\|\mu\|^2_g=\|\mu-\mu'\|^2_\alpha=\|\mu\|^2_\alpha-\|\mu'\|^2_\alpha.\end{equation}
\end{lem}

\begin{lem} [{\rm see \cite[Lemma~3.4]{FZ-Pot}}]\label{eq-r-g} Let\/ $A_1$ be a relatively closed subset
of\/ $D$ with the separation property\/ {\rm(\ref{dist})}. Then a
bounded measure\/ $\mu\in\mathfrak M^+(A_1;D)$ has finite\/
$E_g(\mu)$ if and only if its extension has finite standard
$\alpha$-Riesz energy, and in the affirmative case\/ {\rm(\ref{gr})}
holds. Furthermore, $c_g(A_1)<\infty$ if and only if\/
$c_\alpha(A_1)<\infty$.\end{lem}

\section{Auxiliary results}\label{sec-aux}

In all that follows fix a generalized condenser $\mathbf A=(A_1,A_2)$ (see the Introduction).
Write
\[\mathfrak M(\mathbf A,\mathbf 1):=\bigl\{\mu\in\mathfrak M(\mathbf A): \ \mu^+(A_1)=\mu^-(A_2)=1\bigr\},\]
where $\mathbf 1:=(1,1)$. \textsl{To avoid trivialities, throughout
the paper we assume that}
\begin{equation}\label{nonzero}c_\alpha(A_i)>0\text{ \ \textsl{for\/} }i=1,2.\end{equation}
Then $\mathcal E_\alpha(\mathbf A,\mathbf 1):=\mathcal
E_\alpha(\mathbb R^n)\cap \mathfrak M(\mathbf A,\mathbf 1)$ is
nonempty in accordance with \cite[Lemma~2.3.1]{F1}, and hence so is
$\dot{\mathcal E}_\alpha(\mathbf A,\mathbf 1):=\dot{\mathcal
E}_\alpha(\mathbb R^n)\cap \mathfrak M(\mathbf A,\mathbf 1)$, see
(\ref{eq}). Note that in general $\mathcal E_\alpha(\mathbf
A,\mathbf 1)$ is a \textsl{proper\/} subset of $\dot{\mathcal
E}_\alpha(\mathbf A,\mathbf 1)$, which is seen from the
counterexample given in \cite[Appendix]{DFHSZ2} and
Theorems~\ref{bal-mass-th}, \ref{thm}; compare with the following
Lemma~\ref{l-st}.

\begin{lem}\label{l-st}If\/ $\mathbf A$ is a\/ {\rm(}standard\/{\rm)} condenser with the separation
property\/ {\rm(\ref{dist})}, then
\[\mathcal E_\alpha(\mathbf A,\mathbf 1)=\dot{\mathcal E}_\alpha(\mathbf A,\mathbf 1).\]
\end{lem}

\begin{proof}Fix $\mu\in\mathfrak M(\mathbf A,\mathbf 1)$. By the Riesz composition identity and Fubini's
theorem,
\begin{align*}\int\kappa_{\alpha/2}\mu^+\kappa_{\alpha/2}\mu^-\,dm
 &=\int\left(\int\kappa_{\alpha/2}(x,y)\,d\mu^+(y)\right)
 \left(\int\kappa_{\alpha/2}(x,z)\,d\mu^-(z)\right)\,dm(x)\\
 &=\int\left(\int|x-y|^{\alpha/2-n}|x-z|^{\alpha/2-n}\,dm(x)\right)\,d(\mu^+\otimes\mu^-)(y,z)\\
 &=\int\left(\int|x-y-z|^{\alpha/2-n}|x|^{\alpha/2-n}\,dm(x)\right)\,d(\mu^+\otimes\mu^-)(y,z)\\
 &=\int\kappa_\alpha(y,z)\,d(\mu^+\otimes\mu^-)(y,z)\leqslant
 \bigl[{\rm dist}\,(A_1,D^c)\bigr]^{\alpha-n}<\infty.
\end{align*}
Assuming now $\mu\in\dot{\mathcal E}_\alpha(\mathbf A,\mathbf 1)$, we thus obtain
\[\bigl(\kappa_{\alpha/2}\mu^+\bigr)^2+\bigl(\kappa_{\alpha/2}\mu^-\bigr)^2=
\bigl(\kappa_{\alpha/2}\mu\bigr)^2
+2\kappa_{\alpha/2}\mu^+\,\kappa_{\alpha/2}\mu^-\in L^1(m),\] which
yields $\kappa_{\alpha/2}\mu^+,\kappa_{\alpha/2}\mu^-\in L^2(m)$.
This means that $\mu^+,\mu^-\in\dot{\mathcal E}^+_\alpha(\mathbb
R^n)$ $\bigl(=\mathcal E^+_\alpha(\mathbb R^n)\bigr)$, and hence
$\mu\in\mathcal E_\alpha(\mathbf A,\mathbf 1)$. Since $\mathcal
E_\alpha(\mathbf A,\mathbf 1)\subset\dot{\mathcal E}_\alpha(\mathbf
A,\mathbf 1)$ by (\ref{eq}), the lemma follows.
\end{proof}

We shall also need the following two assertions, the former being known.

\begin{thm} [{\rm see \cite[Theorem~5.1]{FZ-Pot}}]\label{thm} Let\/ $\mu$ be an extendible\/
{\rm(}signed\/{\rm)} Radon measure on\/ $D$ with\/
$E_g(\mu)<\infty$. Then\/ $\mu-\mu'$ has finite weak\/
$\alpha$-Riesz energy\/ $\dot{E}_\alpha(\mu-\mu')$, and moreover
\begin{equation}\label{2.8} E_g(\mu)=\dot{E}_\alpha(\mu-\mu').
\end{equation}
\end{thm}

\begin{thm}\label{l-aux} Assume that\/ $D^c$ is not\/ $\alpha$-thin at infinity. For any\/
$\mu\in\mathcal E^+_{g}(A_1,1;D)$ there is a sequence\/ $\{\mu
_j\}_{j\in\mathbb N}\in\mathcal E^+_\alpha(A_1,1;\mathbb R^n)$, each\/ $\mu_j$ being compactly supported
in\/ $D$, which with the notations\/ $\nu_j:=\mu_j-\mu_j'$ and\/ $\nu:=\mu-\mu'$ possesses the following
properties:
\begin{itemize}
\item[\rm (a)] $\nu\in\dot{\mathcal E}_\alpha(\mathbf A,\mathbf 1)$ and\/
$\{\nu_j\}_{j\in\mathbb N}\subset\mathcal E_\alpha(\mathbf A,\mathbf 1)$;
\item[\rm (b)] $\|\nu_j-\nu\|^\cdot_\alpha\to0$ as\/ $j\to\infty$;
\item[\rm (c)] $\nu_j^\pm\to\nu^\pm$ vaguely in\/ $\mathfrak M(\mathbb R^n)$ as\/ $j\to\infty$.
\end{itemize}
\end{thm}

\begin{proof}Applying Theorems~\ref{bal-mass-th} and \ref{thm} to the (bounded, and hence extendible)
measure $\mu$, we obtain $\nu:=\mu-\mu'\in\dot{\mathcal E}_\alpha(\mathbf A,\mathbf 1)$,
which is the former relation in~(a).

Choose an increasing sequence $\{K_j\}_{j\in\mathbb N}$ of compact sets with the union $D$ and
write $\tilde{\mu}_j:=\mu|_{K_j}$, where $\mu|_{K_j}$ is the trace of $\mu$ on $K_j$. It follows from
the definition of $\tilde{\mu}_j$ that $\kappa_\alpha\tilde{\mu}_j\uparrow\kappa_\alpha\mu$ pointwise
on $\mathbb R^n$ and also that the increasing sequence $\{\tilde{\mu}_j\}_{j\in\mathbb N}$ converges to
$\mu$ vaguely in $\mathfrak M(\mathbb R^n)$. We therefore see from the proof of
\cite[Theorem~3.6]{FZ-Fin} that $\{\tilde{\mu}'_j\}_{j\in\mathbb N}$ likewise is increasing and
converges vaguely to $\mu'$. Also write $\mu_j:=\tilde{\mu}_j/\tilde{\mu}_j(K_j)$, $j\in\mathbb N$.
Since $\tilde{\mu}_j(K_j)\uparrow\mu(A_1)=1$, we infer that $\mu_j\to\mu$ and $\mu_j'\to\mu'$ vaguely
in $\mathfrak M(\mathbb R^n)$, which is~(c).

Furthermore, $g\tilde{\mu}_j\uparrow g\mu$ pointwise on $D$, and also
\begin{equation}\label{upar}\lim_{j\to\infty}\,\|\tilde{\mu}_j\|_g=\|\mu\|_g<\infty.\end{equation}
The former relation implies that $\langle\tilde{\mu}_j,\tilde{\mu}_p\rangle_g
\geqslant\|\tilde{\mu}_p\|^2_g$ for all $j\geqslant p$, and hence
\[\|\tilde{\mu}_j-\tilde{\mu}_p\|^2_g\leqslant\|\tilde{\mu}_j\|^2_g-\|\tilde{\mu}_p\|^2_g,\]
which together with (\ref{upar}) proves that
$\{\tilde{\mu}_j\}_{j\in\mathbb N}$ is a strong Cauchy sequence in
the pre-Hil\-bert space $\mathcal E_g(D)$. Since the kernel $g$ is
perfect \cite[Theorem~4.11]{FZ-Fin}, we thus see by
Definition~\ref{def-perf} that $\tilde{\mu}_j\to\mu$ in $\mathcal
E_g(D)$ strongly, and consequently
\begin{equation}\label{limg}\lim_{j\to\infty}\,\|\mu_j-\mu\|_g=0.\end{equation}

Since $S^{\mu_j}_D$ is compact and $E_g(\mu_j)<\infty$, it follows from Theorem~\ref{bal-mass-th} and
Lemma~\ref{eq-r-g} that $\nu_j:=\mu_j-\mu_j'\in\mathcal E_\alpha(\mathbf A,\mathbf 1)$,
which is the latter relation in (a). Furthermore, we get from Theorem~\ref{thm}
\[\|\nu_j-\nu\|^\cdot_\alpha=\|(\mu_j-\mu)-(\mu_j'-\mu')\|^\cdot_\alpha=\|\mu_j-\mu\|_g,\]
which establishes (b) when combined with (\ref{limg}).
\end{proof}

\section{Statements of the problems}\label{sec-statement}

\subsection{Notations, permanent assumptions, historical remarks}\label{sec-5.1}Let $\mathfrak C(A_1)$
consist of all $\xi\in\mathfrak M^+(A_1;\mathbb R^n)$ with
$\xi(A_1)>1$; such $\xi$ will serve as (upper)
\textsl{constraints\/} for measures of the class $\mathfrak
M^+(A_1,1;\mathbb R^n)$. For any given $\xi\in\mathfrak C(A_1)$
write
\[\dot{\mathcal E}_\alpha^\xi(\mathbf A,\mathbf 1):=
\bigl\{\mu\in\dot{\mathcal E}_\alpha(\mathbf A,\mathbf 1): \
\mu^+\leqslant\xi\bigr\},\] where $\mu^+\leqslant\xi$ means that
$\xi-\mu^+\geqslant0$. To combine (whenever this is possible)
formulations related to minimum weak $\alpha$-Riesz problems in both
the unconstrained and constrained settings, write $\dot{\mathcal
E}_\alpha^\sigma(\mathbf A,\mathbf 1)$, $\sigma\in\mathfrak
C(A_1)\cup\{\infty\}$, where the formal notation $\sigma=\infty$
means that \textsl{no\/} upper constraint is imposed on the positive
parts of $\mu\in\dot{\mathcal E}_\alpha(\mathbf A,\mathbf 1)$, i.e.\
$\dot{\mathcal E}_\alpha^\infty(\mathbf A,\mathbf 1):=\dot{\mathcal
E}_\alpha(\mathbf A,\mathbf 1)$.

Fix $f:\mathbb R^n\to[-\infty,\infty]$, to be treated as an
\textsl{external field\/}, and let $\dot{\mathcal
E}^\sigma_{\alpha,f}(\mathbf A,\mathbf 1)$ consist of all
$\nu\in\dot{\mathcal E}^\sigma_\alpha(\mathbf A,\mathbf 1)$ such
that $f$ is $|\nu|$-int\-egr\-able. For such $\nu$ write $\langle
f,\nu\rangle:=\int f\,d\nu$ and
\begin{equation}\label{G}\dot{G}_{\alpha,f}(\nu):=\|\nu\|^{\cdot\,2}_\alpha+2\langle f,\nu\rangle.
\end{equation}
Note that if $\dot{\mathcal E}^\sigma_{\alpha,f}(\mathbf A,\mathbf 1)$ is nonempty
(sufficient conditions for this to hold will be provided below), then it forms a convex subcone
of the convex cone $\dot{\mathcal E}_\alpha(\mathbf A,\mathbf 1)$.

Fix a nonempty convex cone $\mathcal H\subset\dot{\mathcal E}^\sigma_{\alpha,f}(\mathbf A,\mathbf 1)$,
to be specified below. Then
\[\dot{G}_{\alpha,f}(\mathcal H):=\inf_{\nu\in\mathcal H}\,\dot{G}_{\alpha,f}(\nu)<\infty,\]
and hence the following minimum weak $\alpha$-Riesz energy problem,
to be referred to as the \textsl{$\mathcal H$-problem\/}, makes
sense: \textsl{does there exist\/ $\lambda_{\mathcal H}\in\mathcal
H$ with\/} $\dot{G}_{\alpha,f}(\lambda_{\mathcal
H})=\dot{G}_{\alpha,f}(\mathcal H)$?

\begin{lem}\label{unique}A solution\/ $\lambda_{\mathcal H}$ to the\/ $\mathcal H$-problem is unique\/
{\rm(}whenever it exists\/{\rm)}.\end{lem}

\begin{proof} This can be established by standard methods based on the convexity of the class $\mathcal H$
and the pre-Hil\-bert structure on the space $\dot{\mathcal
E}_\alpha(\mathbb R^n)$. Indeed, if $\lambda$ and $\breve{\lambda}$
are two solutions to the $\mathcal H$-problem, then we obtain from
(\ref{G})
\[4\dot{G}_{\alpha,f}(\mathcal H)\leqslant4\dot{G}_{\alpha,f}\biggl(\frac{\lambda+\breve{\lambda}}{2}\biggr)=
\|\lambda+\breve{\lambda}\|_\alpha^{\cdot\,2}+4\langle f,\lambda+\breve{\lambda}\rangle.\]
On the other hand, applying the parallelogram identity in $\dot{\mathcal E}_\alpha(\mathbb R^n)$ to
$\lambda$ and $\breve{\lambda}$ and then adding and
subtracting $4\langle f,\lambda+\breve{\lambda}\rangle$ we get
\[\|\lambda-\breve{\lambda}\|_\alpha^{\cdot\,2}=
-\|\lambda+\breve{\lambda}\|_\alpha^{\cdot\,2}-4\langle f,\lambda+\breve{\lambda}\rangle+
2\dot{G}_{\alpha,f}(\lambda)+2\dot{G}_{\alpha,f}(\breve{\lambda}).\]
When combined with the preceding relation, this yields
\[0\leqslant\|\lambda-\breve{\lambda}\|^{\cdot\,2}_\alpha\leqslant-4\dot{G}_{\alpha,f}(\mathcal H)+
2\dot{G}_{\alpha,f}(\lambda)+2\dot{G}_{\alpha,f}(\breve{\lambda})=0,\]
which establishes the lemma because $\|\cdot\|^\cdot_\alpha$ is a norm.\end{proof}

\begin{rem}\label{rem2}Let $f=0$, $\sigma=\infty$ (no external field and no constraint), and let
$\mathcal H=\mathcal E_\alpha(\mathbf A,\mathbf 1)$. In view of
(\ref{eq}), the $\mathcal H$-problem is then the problem on the
existence of $\lambda\in\mathcal E_\alpha(\mathbf A,\mathbf 1)$ with
\begin{equation}\label{rem-eq}\|\lambda\|^2_\alpha=
\inf_{\nu\in\mathcal E_\alpha(\mathbf A,\mathbf 1)}\,\|\nu\|^2_\alpha=:w_\alpha(\mathbf A,\mathbf 1),
\end{equation}
$1/w_\alpha(\mathbf A,\mathbf 1)$ being known as the (standard)
$\alpha$-Riesz \textsl{capacity\/} of the (generalized) condenser
$\mathbf A$. To avoid trivialities, assume that
$c_\alpha(A_1)<\infty$.\footnote{If $c_\alpha(A_i)=\infty$ for
$i=1,2$, then $w_\alpha(\mathbf A,\mathbf 1)=0$, and hence this
infimum cannot be an actual minimum because $0\notin\mathcal
E_\alpha(\mathbf A,\mathbf 1)$.} If moreover the separation
condition (\ref{dist}) holds, then problem (\ref{rem-eq}) is
solvable if and only if either $c_\alpha(A_2)<\infty$, or $A_2$ is
\textsl{not\/} $\alpha$-thin at infinity (see \cite[Theorem~5]{ZR}).

Thus, if $A_2$ is $\alpha$-thin at infinity, but $c_\alpha(A_2)=\infty$ (such $A_2$ exists
according to Remark~\ref{rem-thin}), then $\|\nu\|^2_\alpha>w_\alpha(\mathbf A,\mathbf 1)$ for any
$\nu\in\mathcal E_\alpha(\mathbf A,\mathbf 1)$. It can however be shown that any (minimizing) sequence
$\{\nu_j\}_{j\in\mathbb N}\subset\mathcal E_\alpha(\mathbf A,\mathbf 1)$ with
$\lim_{j\to\infty}\,\|\nu_j\|^2_\alpha=w_\alpha(\mathbf A,\mathbf 1)$ then converges strongly and vaguely
to a (unique) measure $\theta\in\mathcal E_\alpha(\mathbf A)$ such that $\theta^+(A_1)=1$, but
$\theta^-(A_2)<1$ (see \cite{ZR}). Of course, $\|\theta\|^2_\alpha=w_\alpha(\mathbf A,\mathbf 1)$.
Using the electrostatic interpretation, which is possible for the Coulomb kernel $|x-y|^{-1}$ on
$\mathbb R^3$, we say that the described pair $(A_1,A_2)$ of oppositely charged conductors achieves
its equilibrium state only provided that a nonzero part (with mass $1-\theta^-(A_2)>0$) of charge
carried by $A_2$ vanishes at the point at infinity.

This phenomenon, discovered first for $\alpha=2$ in \cite{Z0}, is
actually a characteristic feature of \textsl{space\/} condensers;
compare with Bagby's study \cite{Bagby} where it has been proven
that the infimum of the logarithmic energy over $\mathfrak M(\mathbf
A,\mathbf 1)$, $\mathbf A=(A_1,A_2)$ being a condenser in $\mathbb
R^2$ such that $\mathrm{Cl}_{\overline{\mathbb R^2}}\,A_1\cap
\mathrm{Cl}_{\overline{\mathbb R^2}}\,A_2=\varnothing$, is
\textsl{always\/} an actual minimum. (Here $\overline{\mathbb R^2}$
is the one-point compactification of $\mathbb R^2$.) Such a drastic
difference between the theories of space and plane condensers is
caused by the fact that the logarithmic capacity of a plane
condenser is invariant with respect to the M\"{o}bius
transformations (more generally, conformal mappings) of
$\overline{\mathbb R^2}$, while the Riesz capacity of a space
condenser is not so.\end{rem}

\begin{exmp}Let $n=3$, $\alpha=2$ and $A_2=D^c=Q_\varrho$, where $Q_\varrho$ is given by (\ref{descr}),
and let $A_1$ be a closed set in $\mathbb R^3$ with
$c_2(A_1)<\infty$ possessing the separation property (\ref{dist}).
According to \cite[Theorem~5]{ZR}, a solution to problem
(\ref{rem-eq}) (with $\alpha=2$) does exist if $\varrho$ is given by
either (\ref{c1}) or (\ref{c3}), while the problem has no solution
if $\varrho$ is defined by (\ref{c2}). These (theoretical) results
have been illustrated in \cite{HWZ,OWZ} by means of numerical
experiments.\end{exmp}

\textsl{In all that follows we shall always suppose that\/ $D^c$ is
not\/ $\alpha$-thin at infinity\/.}

\begin{rem}\label{rem-FZ-Pot}Assume that $f=0$ and $\sigma=\infty$, and define
\[\ddot{\mathcal E}_\alpha(\mathbf A,\mathbf 1):=\dot{\mathcal E}_\alpha(\mathbf A,\mathbf 1)\cap
\mathrm{Cl}_{\dot{\mathcal E}_\alpha(\mathbb R^n)}\,\mathcal
E_\alpha(\mathbf A,\mathbf 1).\] The $\mathcal H$-problem with
$\mathcal H=\ddot{\mathcal E}_\alpha(\mathbf A,\mathbf 1)$ is in
fact \cite[Problem~6.2]{FZ-Pot}, solved by \cite[Theorems~6.1,
6.2]{FZ-Pot}. To be precise, it has been proven in
\cite[Theorem~6.1]{FZ-Pot} that
\[\inf_{\mu\in\ddot{\mathcal E}_\alpha(\mathbf A,\mathbf 1)}\,\|\mu\|^{\cdot\,2}_\alpha=
\inf_{\nu\in\mathcal E_\alpha(\mathbf A,\mathbf 1)}\,\|\nu\|^2_\alpha\quad
\bigl({}=w_\alpha(\mathbf A,\mathbf 1)\bigr).\]
Furthermore, the $\ddot{\mathcal E}_\alpha(\mathbf A,\mathbf 1)$-problem is shown to be (uniquely) solvable
if and only if
\begin{equation}\label{gcapfinite}c_g(A_1)<\infty,\end{equation}
and the solution $\lambda_{\mathcal H}$ is then given by
\begin{equation}\label{twodotss}\lambda_{\mathcal H}=\mu_{A_1,g}-\mu_{A_1,g}',\end{equation}
where $\mu_{A_1,g}$ is the $g$-capacitary measure on $A_1$ (which
exists,\footnote{It has been used here that $c_g(A_1)>0$, which is
clear from the permanent assumption (\ref{nonzero}) and
footnote~\ref{foot-ga}.\label{cgn}} see Remark~\ref{rem-eq-capacit}
for $F=A_1$). If moreover the separation condition (\ref{dist})
holds, then assumption (\ref{gcapfinite}) (which by
Lemma~\ref{eq-r-g} is now equivalent to $c_\alpha(A_1)<\infty$) is
also necessary and sufficient for the solvability of problem
(\ref{rem-eq}), and its solution is again given by (\ref{twodotss})
\cite[Theorem~6.2]{FZ-Pot}.

Combining Theorems~\ref{bal-mass-th}, \ref{thm}, and~\ref{l-aux}
implies that the quoted results on the solvability of the
$\ddot{\mathcal E}_\alpha(\mathbf A,\mathbf 1)$-prob\-lem remain
valid if $\ddot{\mathcal E}_\alpha(\mathbf A,\mathbf 1)$ is replaced
by the class of all $\mu\in\dot{\mathcal E}_\alpha(\mathbf A,\mathbf
1)$ such that $\mu^+\in\mathcal E_g^+(A_1,1;D)$ and
$\mu^-=(\mu^+)'$. This observation may be viewed as a motivation to
the study of the $\widetilde{\mathcal H}$-problem,
$\widetilde{\mathcal H}$ being defined by (\ref{def-H}) below.
\end{rem}

\textsl{In all that follows we assume that either Case\/~{\rm I} or
Case\/~{\rm II} holds, where:}
\begin{itemize}
\item[\rm I.] \textsl{$f\geqslant0$ is l.s.c.\ on\/ $\mathbb R^n$ and}
\begin{equation}\label{f}
f=0\text{ \ \textsl{n.e.\ on\/} $D^c$;}
\end{equation}
\item[\rm II.] \textsl{$f=\kappa_\alpha(\zeta-\zeta')$ where\/ $\zeta$ is an extendible\/ {\rm(}signed\/{\rm)}
Radon measure on\/ $D$ with\/ $E_g(\zeta)<\infty$ and\/ $\zeta'$ is
the\/ $\alpha$-Riesz swept measure of\/ $\zeta$ onto\/~$D^c$}.
\end{itemize}
Note that in Case~II the external field $f$ is finite n.e.\ on
$\mathbb R^n$ according to the general convention (\ref{1.3.10}),
and it also satisfies (\ref{f}) by (\ref{bal-eq}). In fact, then
(see Lemma~\ref{l-hatg})
\begin{equation}\label{fII}f=g\zeta\text{ \ n.e.\ on\ }D.\end{equation}
Also observe that in Case I we actually have $f=0$ on $D^c$, for
$f\geqslant0$ is l.s.c.

\subsection{An auxiliary minimum $\alpha$-Green energy problem}\label{sec-5.2} Being $c_g$-absolutely
continuous, any $\mu\in\mathcal E^+_g(A_1,1;D)$ is
\textsl{$c_\alpha$-ab\-sol\-ut\-ely continuous\/} (see
footnote~\ref{foot-ga}), which will be used permanently throughout
the paper.

Let $f$ and $\sigma\in\mathfrak C(A_1)\cup\{\infty\}$ be as indicated in Section~\ref{sec-5.1}, and let
$\mathcal E^\sigma_{g,f}(A_1,1;D)$ consist of all $\mu\in\mathcal E^+_g(A_1,1;D)$ such that
$\mu\leqslant\sigma$ and $f$ is $\mu$-int\-egr\-able.  For any $\mu\in\mathcal E^\sigma_{g,f}(A_1,1;D)$ write
\begin{equation}\label{def-G}G_{g,f}(\mu):=\|\mu\|^2_g+2\langle f,\mu\rangle=
|\mu\|^2_g+2\langle f|_D,\mu\rangle,\end{equation}
the latter equality being valid since $\mu^*(D^c)=0$ for any $\mu\in\mathfrak M^+(D;\mathbb R^n)$,
see Section~\ref{sec:princ}. If the class $\mathcal E^\sigma_{g,f}(A_1,1;D)$ is nonempty, or equivalently
\begin{equation}\label{gfin}G^\sigma_{g,f}(A_1,1;D):=
\inf_{\mu\in\mathcal
E^\sigma_{g,f}(A_1,1;D)}\,G_{g,f}(\mu)<\infty,\end{equation} then
the following (auxiliary) minimum $\alpha$-Green energy problem,
which is actually \cite[Problem~3.2]{DFHSZ}, makes sense.

\begin{problem}\label{3.2}Does there exist\/ $\lambda_{A_1,g}\in\mathcal E^\sigma_{g,f}(A_1,1;D)$ with
\[G_{g,f}(\lambda_{A_1,g})=G^\sigma_{g,f}(A_1,1;D)?\]
\end{problem}

\begin{lem}\label{minusinfty} $G^\sigma_{g,f}(A_1,1;D)>-\infty$.
\end{lem}

\begin{proof} This is obvious in Case~I since, by the strict positive definiteness of the kernel~$g$,
\[G_{g,f}(\mu)=\|\mu\|^2_g+2\langle f,\mu\rangle>0\text{ \ for every\ }\mu\in\mathcal E^\sigma_{g,f}(A_1,1;D).
\]

If Case II takes place, then by (\ref{fII}) the equality $f=g\zeta$
holds n.e.\ on $D$, and hence $\mu$-a.e.\ for any
($c_\alpha$-ab\-sol\-utely continuous) measure $\mu\in\mathcal
E^\sigma_{g,f}(A_1,1;D)$. Integrating $f=g\zeta$ with respect to
$\mu$ gives $2\langle
f,\mu\rangle=2E_g(\zeta,\mu)=\|\mu+\zeta\|^2_g-\|\mu\|^2_g-\|\zeta\|^2_g$,
and therefore
\begin{equation}\label{caseII}G_{g,f}(\mu)=\|\mu+\zeta\|^2_g-\|\zeta\|^2_g\geqslant-\|\zeta\|^2_g>-\infty.
\end{equation}
This completes the proof by letting here $\mu$ range over $\mathcal E^\sigma_{g,f}(A_1,1;D)$.\end{proof}

\textsl{In all that follows we assume that}
\begin{equation}\label{suff}c_g(A_1^\circ)>0,\text{ \textsl{where\/} }A_1^\circ:=\bigl\{x\in A_1: \
|f(x)|<\infty\bigr\}.\end{equation} Note that in Case II this holds
automatically, which is clear from $c_g(A_1)>0$ (see
footnote~\ref{cgn}) and the fact that in Case~II the external field
is finite n.e.\ on $\mathbb R^n$.

\begin{lem}\label{suff-fin} $G^\sigma_{g,f}(A_1,1;D)$ is finite if either\/ $\sigma=\infty$,
or otherwise if the following two requirements hold: $\sigma(A_1^\circ)>1$ and\/ $E_g(\sigma|_K)<\infty$
for every compact\/ $K\subset A_1^\circ$.
\end{lem}

\begin{proof}As seen from Lemma~\ref{minusinfty}, it is enough to show that under the stated assumptions
$\mathcal E_{g,f}^\sigma(A_1,1;D)$ is nonempty. Write $E_k:=\{x\in
A_1: \ |f(x)|\leqslant k\}$. As $E_k$, $k\in\mathbb N$, are
universally measurable and $E_k\uparrow A_1^\circ$, we get from
\cite[Lemma~2.3.3]{F1}
\begin{equation}\label{ek}c_g(A_1^\circ)=\lim_{k\to\infty}\,c_g(E_k).\end{equation}

Assume first that $\sigma=\infty$. Since $c_g(A_1^\circ)>0$ by
assumption, in view of (\ref{ek}) and (\ref{compact}) one can choose
$k_0\in\mathbb N$ and a compact set $K_0\subset E_{k_0}$ with
$c_g(K_0)>0$, and hence a measure $\mu\in\mathcal E_g^+(K_0,1;D)$.
Noting that $|\langle f,\mu\rangle|\leqslant k_0$, we actually have
$\mu\in\mathcal E_{g,f}^\sigma(A_1,1;D)$.\footnote{These arguments
can actually be reversed, which proves that in the unconstrained
case ($\sigma=\infty$) the permanent assumption (\ref{suff}) is
necessary and sufficient for the finiteness of
$G^\sigma_{g,f}(A_1,1;D)$.}

Consider next the case where $\sigma\in\mathfrak C(A_1)$. Then $\sigma(A_1^\circ)>1$ by assumption,
and hence there exist $k\in\mathbb N$ and a compact set $K\subset E_k$ such that $\sigma(K)>1$.
Since $E_g(\sigma|_K)<\infty$ and $|f|\leqslant k$ on $K$, the measure $\sigma|_K/\sigma(K)$ belongs
to $\mathcal E_{g,f}^\sigma(A_1,1;D)$.
\end{proof}

\textsl{Unless explicitly stated otherwise, in all that follows\/
$G^\sigma_{g,f}(A_1,1;D)$ is required to be finite.} Sufficient
conditions for this to hold have been provided in
Lemma~\ref{suff-fin}.

\begin{lem} [{\rm see \cite[Lemma~4.1]{DFHSZ}}] A solution $\lambda_{A_1,g}$ to
Problem\/~{\rm\ref{3.2}}
is unique\/ {\rm(}if it exists\/{\rm)}.\end{lem}

\begin{rem}\label{rem-zero-inf}If $f=0$ and $\sigma=\infty$, then Problem~\ref{3.2} reduces to
problem (\ref{cap-def}) with $Q=A_1$ and $\kappa=g$. It therefore
has the (unique) solution $\lambda_{A_1,g}$ if and only if condition
(\ref{gcapfinite}) holds, and then (see Remark~\ref{rem-eq-capacit})
\begin{equation}\label{eq-capacit}\lambda_{A_1,g}=\mu_{A_1,g}=\gamma_{A_1,g}/c_g(A_1),\end{equation}
where $\mu_{A_1,g}$, resp.\ $\gamma_{A_1,g}$, is the $g$-capacitary,
resp.\ $g$-equilibrium, measure on $A_1$.\end{rem}

\subsection{A minimum weak $\alpha$-Riesz energy problem for a condenser with separated plates}\label{sec-5.3}
Throughout Section~\ref{sec-5.3}, $\mathbf A$ is a (standard) condenser possessing the separation
property (\ref{dist}). Denoting by $\mathcal E^\sigma_{\alpha,f}(\mathbf A,\mathbf 1)$ the class of all
$\nu\in\mathcal E_\alpha(\mathbf A,\mathbf 1)$ such that $\nu^+\leqslant\sigma$
and $f$ is $|\nu|$-int\-egr\-able, we then obtain from Lemma~\ref{l-st}
\begin{equation}\label{but}\mathcal E^\sigma_{\alpha,f}(\mathbf A,\mathbf 1)=
\dot{\mathcal E}^\sigma_{\alpha,f}(\mathbf A,\mathbf
1).\end{equation} (In general, $\mathcal E^\sigma_{\alpha,f}(\mathbf
A,\mathbf 1)$ is a \textsl{proper\/} subset of $\dot{\mathcal
E}^\sigma_{\alpha,f}(\mathbf A,\mathbf 1)$.) For any $\nu\in\mathcal
E^\sigma_{\alpha,f}(\mathbf A,\mathbf 1)$ write
\begin{equation}\label{def-G-a}G_{\alpha,f}(\nu):=\|\nu\|^2_\alpha+2\langle f,\nu\rangle=
\|\nu\|^{\cdot\,2}_\alpha+2\langle
f,\nu\rangle=\dot{G}_{\alpha,f}(\nu),\end{equation} the last two
equalities being obtained from (\ref{eq}) and~(\ref{G}).

As will be shown at the beginning of the proof of Theorem~\ref{th-equiv}, for any
$\mu\in\mathcal E_{g,f}^\sigma(A_1,1;D)$ (which exists by the permanent assumption (\ref{gfin}))
we have $\mu-\mu'\in\mathcal E^\sigma_{\alpha,f}(\mathbf A,\mathbf 1)$. Hence
\[G_{\alpha,f}^\sigma(\mathbf A,\mathbf 1):=
\inf_{\nu\in\mathcal E^\sigma_{\alpha,f}(\mathbf A,\mathbf
1)}\,G_{\alpha,f}(\nu)<\infty,\] and the following minimum standard
$\alpha$-Riesz energy problem, which is actually
\cite[Problem~3.1]{DFHSZ}, therefore makes sense.

\begin{problem}\label{3.1} Does there exist\/ $\lambda_{\mathbf
A,\alpha}\in\mathcal E^\sigma_{\alpha,f}(\mathbf A,\mathbf 1)$ with
\[G_{\alpha,f}(\lambda_{\mathbf
A,\alpha})=G_{\alpha,f}^\sigma(\mathbf A,\mathbf 1)?\]\end{problem}

By (\ref{but}) and (\ref{def-G-a}), this problem is in fact the
$\mathcal H$-problem with $\mathcal H=\dot{\mathcal
E}^\sigma_{\alpha,f}(\mathbf A,\mathbf 1)$, and hence a solution to
Problem~\ref{3.1} is \textsl{unique\/} provided that it exists (see
Lemma~\ref{unique}).

\begin{thm}\label{th-equiv}If the separation condition\/ {\rm(\ref{dist})} holds, then in
 Cases\/~{\rm I} and\/~{\rm II}
\begin{equation}\label{eq-1-m}G_{\alpha,f}^\sigma(\mathbf A,\mathbf 1)=G_{g,f}^\sigma(A_1,1;D),\end{equation}
$\sigma\in\mathfrak C(A_1)\cup\{\infty\}$ being arbitrary.
Furthermore, Problem\/~{\rm\ref{3.1}} is\/ {\rm(}uniquely\/{\rm)}
solvable if and only if Problem\/~{\rm\ref{3.2}} is so, and then
their solutions are related to one another as follows:
\begin{equation}\label{eq-2-m}\lambda_{\mathbf A,\alpha}=\lambda_{A_1,g}-\lambda_{A_1,g}'.\end{equation}
\end{thm}

\begin{proof}Any measure $\mu\in\mathcal E_{g,f}^\sigma(A_1,1;D)$ has finite standard $\alpha$-Riesz energy,
which is clear from the separation condition (\ref{dist}) by
Lemma~\ref{eq-r-g}, and hence so does the $\alpha$-Riesz swept
measure $\mu'$ of $\mu$ onto $A_2=D^c$. Furthermore, $\mu'(A_2)=1$
according to Theorem~\ref{bal-mass-th}, and so
$\nu:=\mu-\mu'\in\mathcal E^\sigma_{\alpha,f}(\mathbf A,\mathbf 1)$.
We therefore get from (\ref{def-G}) and~(\ref{def-G-a})
\begin{equation}\label{prr-1}G_{g,f}(\mu)=\|\mu\|^2_g+2\langle f,\mu\rangle=\|\mu-\mu'\|^2_\alpha+2\langle
f,\mu-\mu'\rangle=G_{\alpha,f}(\nu)\geqslant
G_{\alpha,f}^\sigma(\mathbf A,\mathbf 1),\end{equation} the second
equality being valid by Lemma~\ref{eq-r-g}, (\ref{f}), and the
$c_\alpha$-absolute continuity of $\mu'$
\cite[Corollary~3.19]{FZ-Fin}. By letting here $\mu$ vary over
$\mathcal E_{g,f}^\sigma(A_1,1;D)$, we obtain
\begin{equation}\label{inequality1}G_{g,f}^\sigma(A_1,1;D)\geqslant
G_{\alpha,f}^\sigma(\mathbf A,\mathbf 1).\end{equation}

Further, fix $\nu\in\mathcal E^\sigma_{\alpha,f}(\mathbf A,\mathbf 1)$. Then
$\nu^+\in\mathcal E_{g,f}^\sigma(A_1,1;D)$, for $E_g(\nu^+)\leqslant E_\alpha(\nu^+)$. Thus
\begin{align}\label{a}G_{\alpha,f}(\nu)=\|\nu\|^2_\alpha+2\langle f,\nu^+\rangle&
\geqslant\|\nu^+-(\nu^+)'\|^2_\alpha+2\langle f,\nu^+\rangle\\
{}&=\|\nu^+\|^2_g+2\langle f,\nu^+\rangle\geqslant
G_{g,f}^\sigma(A_1,1;D),\notag\end{align} where the former equality
holds by (\ref{f}) and the $c_\alpha$-absolute continuity of
$\nu^-$, the former inequality is obtained from (\ref{proj}), and
the latter equality is valid according to Lemma~\ref{l-hen'}.
Letting here $\nu$ range over $\mathcal E^\sigma_{\alpha,f}(\mathbf
A,\mathbf 1)$ and then combining the inequality obtained with
(\ref{inequality1}), we arrive at~(\ref{eq-1-m}).

Suppose now that $\lambda_{\mathbf A,\alpha}$ solves
Problem~\ref{3.1}. Substituting $\lambda_{\mathbf A,\alpha}$ in
place of $\nu$ into (\ref{a}) and then combining this with
(\ref{eq-1-m}), we see that in fact equality prevails in either of
the two inequalities. Problem~\ref{3.2} has therefore a solution
$\lambda_{A_1,g}$; and moreover $\lambda_{\mathbf
A,\alpha}^+=\lambda_{A_1,g}$ and $\lambda_{\mathbf
A,\alpha}^-=(\lambda_{\mathbf A,\alpha}^+)'=\lambda_{A_1,g}'$, which
establishes~(\ref{eq-2-m}).

To complete the proof, assume finally that
$\lambda_{A_1,g}\in\mathcal E_{g,f}^\sigma(A_1,1;D)$ solves
Problem~\ref{3.2}. Substituting $\lambda_{A_1,g}$ in place of $\mu$
into (\ref{prr-1}) and then combining this with (\ref{eq-1-m}), we
see that the inequality in the relation thus obtained is in fact
equality. Problem~\ref{3.1} has therefore a solution
$\lambda_{\mathbf A,\alpha}$, and moreover $\lambda_{\mathbf
A,\alpha}=\lambda_{A_1,g}-\lambda_{A_1,g}'$.
\end{proof}

\textsl{In all that follows, when speaking of
Problem\/~{\rm\ref{3.1}} we shall always tacitly assume the
separation condition\/ {\rm(\ref{dist})} to hold.} Note that in the
case where $f=0$ and $\sigma=\infty$, Problem~\ref{3.1} reduces to
problem (\ref{rem-eq}), solved by \cite[Theorem~6.2]{FZ-Pot} quoted
in Remark~\ref{rem-FZ-Pot} above.

As seen from (\ref{eq-2-m}), Theorem~\ref{th-equiv} remains in force
if the class $\mathcal E^\sigma_{\alpha,f}(\mathbf A,\mathbf 1)$ of
admissible measures in Problem~\ref{3.1} is replaced by that of
$\mu-\mu'$, where $\mu$ ranges over $\mathcal
E^\sigma_{g,f}(A_1,1;D)$. We are thus led to the study of the
$\widetilde{\mathcal H}$-problem, $\widetilde{\mathcal H}$ being
defined by (\ref{def-H}) below.

\subsection{A minimum weak $\alpha$-Riesz energy problem for a generalized condenser}\label{sec-5.4}
Omitting now the separation condition (\ref{dist}), we shall further show that the equivalence of
\cite[Problem~3.1]{DFHSZ} and \cite[Problem~3.2]{DFHSZ}, established in Theorem~\ref{th-equiv} above,
remains valid if the concept of standard $\alpha$-Riesz energy in \cite[Problem~3.1]{DFHSZ} is replaced
by that of weak $\alpha$-Riesz energy, applied now to the (new) admissible measures
$\nu\in\widetilde{\mathcal H}$ defined as follows:
\begin{equation}\label{def-H}\widetilde{\mathcal H}:=
\bigl\{\nu=\mu-\mu': \ \mu\in\mathcal E^\sigma_{g,f}(A_1,1;D)\bigr\}.\end{equation}

In view of the permanent assumption (\ref{gfin}) the class $\mathcal E^\sigma_{g,f}(A_1,1;D)$ is nonempty,
and hence so is $\widetilde{\mathcal H}$, which obviously forms a convex cone. Furthermore,
\[\widetilde{\mathcal H}\subset\dot{\mathcal E}^\sigma_{\alpha,f}(\mathbf A,\mathbf 1),\]
which is seen from Theorems~\ref{bal-mass-th}, \ref{thm} and the
equality $\langle f,\mu'\rangle=0$, the last being obtained from
(\ref{f}) in view of the $c_\alpha$-ab\-sol\-ute continuity of
$\mu'$ \cite[Corollary~3.19]{FZ-Fin}.

For any $\nu=\mu-\mu'\in\widetilde{\mathcal H}$ we therefore get
from (\ref{G}), (\ref{2.8}), and~(\ref{def-G})
\[\dot{G}_{\alpha,f}(\nu)=\|\nu\|^{\cdot\,2}_\alpha+2\langle f,\nu\rangle=
\|\mu\|^2_g+2\langle f,\mu\rangle=G_{g,f}(\mu),\]
and hence
\[\dot{G}_{\alpha,f}(\widetilde{\mathcal H}):=
\inf_{\nu\in\widetilde{\mathcal H}}\,\dot{G}_{\alpha,f}(\nu)=G^\sigma_{g,f}(A_1,1;D).\]

Consider the \textsl{$\widetilde{\mathcal H}$-problem on the
existence of\/ $\lambda_{\widetilde{\mathcal
H}}\in\widetilde{\mathcal H}$ with\/}
$\dot{G}_{\alpha,f}(\lambda_{\widetilde{\mathcal
H}})=\dot{G}_{\alpha,f}(\widetilde{\mathcal H})$. According to
Lemma~\ref{unique} with $\mathcal H=\widetilde{\mathcal H}$, a
solution $\lambda_{\widetilde{\mathcal H}}$ is unique (if it
exists).

Summarizing what we have thus observed, we arrive at the following conclusion.

\begin{thm}\label{lem-equ}In both Cases\/~{\rm I} and\/~{\rm II} for any\/
$\sigma\in\mathfrak C(A_1)\cup\{\infty\}$ we have
\[\dot{G}_{\alpha,f}(\widetilde{\mathcal H})=G_{g,f}^\sigma(A_1,1;D).\]
The\/ $\widetilde{\mathcal H}$-problem has the\/
{\rm(}unique\/{\rm)} solution\/ $\lambda_{\widetilde{\mathcal H}}$
if and only if there is the\/ {\rm(}unique\/{\rm)} solution\/
$\lambda_{A_1,g}$ to Problem\/~{\rm\ref{3.2}}, and in the
affirmative case the following formula holds:
\begin{equation}\label{eq-incl}\lambda_{\widetilde{\mathcal H}}=\lambda_{A_1,g}-\lambda_{A_1,g}'.\end{equation}
\end{thm}

\subsection{The case where $f=0$ and $\sigma=\infty$} We shall need the following particular case of
Theorems~\ref{th-equiv} and~\ref{lem-equ}, which can also be obtained from \cite[Theorems~6.1, 6.2]{FZ-Pot}
quoted in Remark~\ref{rem-FZ-Pot} above.

\begin{thm}\label{cor-un-un}Let\/ $f=0$ and\/ $\sigma=\infty$ both hold, and let\/
$\nu\in\widetilde{\mathcal H}$, resp.\ $\nu\in\mathcal
E^\sigma_{\alpha,f}(\mathbf A,\mathbf 1)$, be given. Then\/ $\nu$
solves the\/ $\widetilde{\mathcal H}$-problem, resp.\
Problem\/~{\rm\ref{3.1}}, if and only if assumption\/
{\rm(\ref{gcapfinite})} holds, and in the affirmative case we have\/
\[\nu=\mu_{A_1,g}-\mu_{A_1,g}'=(\gamma_{A_1,g}-\gamma_{A_1,g}')/c_g(A_1).\]
\end{thm}

\begin{proof}By Theorem~\ref{lem-equ}, resp.\ Theorem~\ref{th-equiv}, $\nu$ solves the
$\widetilde{\mathcal H}$-problem, resp.\ Problem~\ref{3.1}, if and
only if $\nu^+$ solves Problem~\ref{3.2}, which in turn holds if and
only if $c_g(A_1)$ is finite, and then $\nu^+=\lambda_{A_1,g}$ is in
fact the $g$-capacitary measure $\mu_{A_1,g}$ on $A_1$ (see
Remark~\ref{rem-zero-inf}). Substituting (\ref{eq-capacit}) into
(\ref{eq-incl}), resp.\ (\ref{eq-2-m}), completes the
proof.\end{proof}

\section{On the existence of minimizers}

Let a generalized condenser $\mathbf A=(A_1,A_2)$, a constraint
$\sigma\in\mathfrak C(A_1)\cup\{\infty\}$, and an external field $f$
satisfy all the permanent assumptions indicated in
Sections~\ref{sec-aux} and~\ref{sec-statement}. Recall that when
speaking of Problem~\ref{3.1}, we always tacitly assume the
separation condition (\ref{dist}) to hold.

Theorems~\ref{main-th}, \ref{infcap} and Corollary~\ref{cor-n-s}
below provide sufficient and/or necessary conditions for the
solvability of the $\widetilde{\mathcal H}$-problem as well as
Problem~\ref{3.1}. See also Theorems~\ref{desc-th1}, \ref{cor-desc},
\ref{desc-infty}, and \ref{cor-infty-2} in Section~\ref{desc} where
such criteria are established in terms of variational inequalities
for the $\alpha$-Riesz potentials.

\begin{thm}\label{main-th}The\/ $\widetilde{\mathcal H}$-problem as well as Problem\/~{\rm\ref{3.1}}
is\/ {\rm(}uniquely\/{\rm)} solvable in both Cases\/~{\rm I} and\/~{\rm II} provided that\/ $c_g(A_1)<\infty$
or\/ $\sigma(A_1)<\infty$ holds.
\end{thm}

\begin{proof}According to Theorems~\ref{th-equiv} and \ref{lem-equ}, it is enough to show that under the
stated assumptions Problem~\ref{3.2} is
solvable.\footnote{Sufficient conditions for the solvability of
Problem~\ref{3.2}, established in Theorem~\ref{main-th}, are less
demanding than those in \cite{DFHSZ} (cf.\ \cite[Theorems~5.1,
5.2]{DFHSZ}), since our present proof is based on the perfectness of
the $\alpha$-Green kernel and the existence of the $\alpha$-Green
equilibrium measure \cite[Theorem~4.11, 4.12]{FZ-Fin}.}

A sequence $\{\mu_k\}_{k\in\mathbb N}\subset\mathcal
E_{g,f}^\sigma(A_1,1;D)$ is said to be \textsl{minimizing\/} in
Problem~\ref{3.2} if
\begin{equation}\label{minimiz}\lim_{k\to\infty}\,G_{g,f}(\mu_k)=G^\sigma_{g,f}(A_1,1;D).\end{equation}
Let $\mathbb M^\sigma_{g,f}(A_1,1;D)$ consist of all these
$\{\mu_k\}_{k\in\mathbb N}$, which exist by the permanent
assumption~(\ref{gfin}). We proceed the proof of
Theorem~\ref{main-th} with the following auxiliary assertions.

\begin{lem}\label{fund-l} For any\/ $\{\mu_k\}_{k\in\mathbb N}$ and\/ $\{\nu_k\}_{k\in\mathbb N}$
in\/ $\mathbb M^\sigma_{g,f}(A_1,1;D)$,
\begin{equation}\label{fund-l-e}\lim_{k\to\infty}\,\|\mu_k-\nu_k\|_g=0.\end{equation}
\end{lem}

\begin{proof}Based on the convexity of the class $\mathcal E^\sigma_{g,f}(A_1,1;D)$ and the
pre-Hil\-bert structure on the space $\mathcal E_g(D)$, similarly as
in the proof of Lemma~\ref{unique} we obtain
\[0\leqslant\|\mu_k-\nu_k\|^2_g\leqslant-4G_{g,f}^\sigma(A_1,1;D)+2G_{g,f}(\mu_k)+2G_{g,f}(\nu_k).\]
Since $-\infty<G_{g,f}^\sigma(A_1,1;D)<\infty$, we get
(\ref{fund-l-e}) from (\ref{minimiz}) by letting here $k\to\infty$.
\end{proof}

\begin{cor}\label{fund-c}Every\/ $\{\mu_k\}_{k\in\mathbb N}\in\mathbb M^\sigma_{g,f}(A_1,1;D)$ is
a strong Cauchy sequence in\/~$\mathcal E^+_g(D)$.\end{cor}

\begin{lem}\label{l-cont}The mapping\/ $\mu\mapsto G_{g,f}(\mu)$ is vaguely l.s.c.\ on\/
$\mathcal E_{g,f}^\sigma(A_1,1;D)$ if Case\/~{\rm I} holds, and in Case\/~{\rm II} it is strongly continuous.
\end{lem}

\begin{proof}Since the mapping $\mu\mapsto E_g(\mu)$ is vaguely l.s.c.\ on $\mathfrak M^+(D)$ according to
\cite[Lemma~2.2.1(e)]{F1} with $\kappa=g$, the former assertion
follows from Lemma~\ref{lemma-semi} with $\psi=f$. The latter
assertion is obtained directly from the equality
in~(\ref{caseII}).\end{proof}

Let $\mathcal E^\sigma_{g,f}(A_1;D)$ consist of all $\nu\in\mathcal E^+_g(A_1;D)$ such that
$\nu\leqslant\sigma$ and $f$ is $\nu$-int\-egr\-able.

\begin{lem}\label{extremal}There is a unique measure\/ $\theta\in\mathcal E^\sigma_{g,f}(A_1;D)$
such that every minimizing sequence\/ $\{\mu_k\}_{k\in\mathbb
N}\in\mathbb M^\sigma_{g,f}(A_1,1;D)$ converges to\/ $\theta$ both
strongly and vaguely in\/ $\mathcal E^+_g(D)$.\end{lem}

\begin{proof}Fix $\{\mu_k\}_{k\in\mathbb N}\in\mathbb M^\sigma_{g,f}(A_1,1;D)$. Being bounded in the vague
topology on $\mathfrak M^+(D)$, $\{\mu_k\}_{k\in\mathbb N}$ has a vague cluster point $\theta$
\cite[Chapter~III, Section~2, Proposition~9]{B2}, and hence there is a subsequence
$\{\mu_{k_j}\}_{j\in\mathbb N}$ of $\{\mu_k\}_{k\in\mathbb N}$ converging vaguely to $\theta$.
Thus $\mu_{k_j}\otimes\mu_{k_j}\to\theta\otimes\theta$ vaguely in $\mathfrak M^+(D\times D)$
\cite[Chapter~III, Section~5, Exercise~5]{B2}. Being strong Cauchy in $\mathcal E^+_g(D)$
according to Corollary~\ref{fund-c}, $\{\mu_k\}_{k\in\mathbb N}$ is strongly bounded.
Applying Lemma~\ref{lemma-semi} to $X=D\times D$ and $\psi=g$, we therefore get
$\theta\in\mathcal E^+_g(D)$. Since $A_1$ is (relatively) closed in $D$, $\theta$ is carried by $A_1$.
Also note that $\theta\leqslant\sigma$, for the vague limit of a sequence of positive measures
likewise is positive.

By the perfectness of the $\alpha$-Green kernel $g$, $\mu_k\to\theta$ strongly in $\mathcal E^+_g(D)$
when $k\to\infty$ (Definition~\ref{def-perf}). As $g$ is strictly positive definite, $\theta$ must be
a unique (strong and) vague cluster point of $\{\mu_k\}_{k\in\mathbb N}$. Since the vague topology is
Hausdorff, a (unique) vague cluster point $\theta$ of $\{\mu_k\}_{k\in\mathbb N}$ must be its vague limit
\cite[Chapter~I, Section~9, n$^\circ$\,1, Corollary]{B1}.

If $\{\nu_k\}_{k\in\mathbb N}$ is another element of $\mathbb
M^\sigma_{g,f}(A_1,1;D)$, then $\nu_k\to\theta$ strongly in
$\mathcal E^+_g(D)$, which is clear from the strong convergence of
$\{\mu_k\}_{k\in\mathbb N}$ to $\theta$ and (\ref{fund-l-e}). Since
the strong topology on $\mathcal E^+_g(D)$ is finer than the vague
topology (Theorem~\ref{fu-complete}), $\theta$ is also the vague
limit of $\{\nu_k\}_{k\in\mathbb N}$.

To complete the proof, it remains to show that $\langle
f,\theta\rangle<\infty$, or equivalently $G_{g,f}(\theta)<\infty$.
As $\mu_k\to\theta$ strongly and vaguely in $\mathcal E^+_g(D)$, it
follows from Lemma~\ref{l-cont} and (\ref{minimiz}) that
\[G_{g,f}(\theta)\leqslant\lim_{k\to\infty}\,G_{g,f}(\mu_k)=G^\sigma_{g,f}(A_1,1;D).\]
Since $G^\sigma_{g,f}(A_1,1;D)$ is finite by the permanent condition
(\ref{gfin}), the lemma follows.
\end{proof}

We are now ready to complete the proof of Theorem~\ref{main-th}. As seen from the last display,
this will be done once we have proven that if $c_g(A_1)<\infty$ or $\sigma(A_1)<\infty$ holds,
then actually $\theta\in\mathcal E^\sigma_{g,f}(A_1,1;D)$, or equivalently
\begin{equation}\label{solution1}\theta(A_1)=1,\end{equation}
$\theta$ being determined uniquely by Lemma~\ref{extremal}.

To this end consider an exhaustion of $A_1$ by an increasing
sequence $\{K_j\}_{j\in\mathbb N}$ of compact sets, and fix
$\{\mu_k\}_{k\in\mathbb N}\in\mathbb M^\sigma_{g,f}(A_1,1;D)$. By
Lemma~\ref{extremal}, then $\mu_k\to\theta$ both strongly and
vaguely in $\mathcal E^+_g(D)$. Since $1_{K_j}$ is upper
semicontinuous (and bounded), while $1_D$ is (finitely) continuous
on $D$, we obtain from Lemma~\ref{lemma-semi} with $X=D$, applied
subsequently to $1_D$ and $-1_{K_j}$,
\begin{align*}1=\lim_{k\to\infty}\,\mu_k(A_1)\geqslant\theta(A_1)&=\lim_{j\to\infty}\,\theta(K_j)\geqslant
\lim_{j\to\infty}\,\limsup_{k\to\infty}\,\mu_k(K_j)\\&{}=
1-\lim_{j\to\infty}\,\liminf_{k\to\infty}\,\mu_k(A_1\setminus
K_j).\end{align*}
Equality (\ref{solution1}) will therefore follow if we prove the relation
\begin{equation}\label{l}\lim_{j\to\infty}\,\liminf_{k\to\infty}\,\mu_k(A_1\setminus
K_j)=0.\end{equation}

Assume first that the constraint $\sigma=\xi\in\mathfrak C(A_1)$ is bounded. Since
\[\infty>\xi(A_1)=\lim_{j\to\infty}\,\xi(K_j),\]
it follows that
\[\lim_{j\to\infty}\,\xi(A_1\setminus K_j)=0.\]
When combined with $\mu_k(A_1\setminus K_j)\leqslant\xi(A_1\setminus
K_j)$ for any $k,j\in\mathbb N$, this implies~(\ref{l}).

Assume finally that $c_g(A_1)<\infty$. For any  $F\subset A_1$, relatively closed in $D$, then
there exists the (unique) $\alpha$-Green equilibrium measure $\gamma_F$ on $F$, see Theorem~\ref{th-equi}.

Write $K_j^*:=\mathrm{Cl}_D(A_1\setminus K_j)$, $j\in\mathbb N$. By
the monotonicity of $K_j^*$ (as $j$ ranges over $\mathbb N$) and
(\ref{geq-1'}) for $F=K_j^*$, the measure $\gamma_j:=\gamma_{K_j^*}$
belongs to $\Gamma_p:=\Gamma_{K_p^*}$ for all $p\geqslant j$,
$\Gamma_F$ being defined in Theorem~\ref{th-equi}, while $\gamma_p$
solves problem (\ref{alt}) with $F=K_p^*$. Applying
\cite[Lemma~4.1.1]{F1} to the convex class $\Gamma_p$, we therefore
obtain
\[
\|\gamma_j-\gamma_p\|^2_g\leqslant\|\gamma_j\|^2_g-\|\gamma_p\|^2_g\text{
\ for all\ }p\geqslant j.\] Furthermore, it is clear from
(\ref{alt}) with $F=K_j^*$ that the sequence
$\bigl\{\|\gamma_j\|^2_g\bigr\}_{j\in\mathbb N}$ is bounded and
decreasing, and hence it is Cauchy in $\mathbb R$. The preceding
inequality thus implies that $\{\gamma_j\}_{j\in\mathbb N}$ is
strong Cauchy in $\mathcal E^+_g(D)$. Since it obviously converges
vaguely to zero in $\mathfrak M^+(D)$, zero is also its strong limit
because of the perfectness of the kernel $g$. Hence,
\[\lim_{j\to\infty}\,\|\gamma_j\|_g=0.\]
Integrating (\ref{geq-1'}) with $F=K^*_j$ with respect to the
measure $\mu_k$ (which is $c_\alpha$-absolutely continuous, see the
beginning of Section~\ref{sec-5.2}) and then applying the
Cauchy--Schwarz (Bunyakovski) inequality in the pre-Hilbert space
$\mathcal E_g(D)$, we get
\[\mu_k(A_1\setminus K_j)\leqslant\mu_k(K^*_j)=E_g(\gamma_j,\mu_k)\leqslant
\|\gamma_j\|_g\|\mu_k\|_g\text{ \ for all\ }k,j\in\mathbb N.\] As
$\{\|\mu_k\|_g\}_{k\in\mathbb N}$ is bounded, combining the last two
relations again results in (\ref{l}), thus completing the proof of
Theorem~\ref{main-th}.
\end{proof}

As seen from the following Theorem~\ref{infcap}, the sufficient
conditions on the solvability, established in Theorem~\ref{main-th},
are actually \textsl{sharp\/}.

\begin{thm}\label{infcap} Suppose Case\/~{\rm II} with\/ $\zeta\geqslant0$ takes place. If moreover
$c_g(A_1)=\infty$, then the\/ $\widetilde{\mathcal H}$-problem as
well as Problem\/~{\rm\ref{3.1}} is unsolvable for every\/
$\sigma\in\mathfrak C(A_1)\cup\{\infty\}$ such that\/
$\sigma\geqslant\xi_0$, where\/ $\xi_0\in\mathfrak C(A_1)$ with\/
$\xi_0(A_1)=\infty$ is chosen properly.
\end{thm}

\begin{proof}According to Theorems~\ref{th-equiv} and \ref{lem-equ}, it is enough to show that
under the stated assumptions Problem~\ref{3.2} is unsolvable. Since
Case~II with $\zeta\geqslant0$ takes place,
\begin{equation}\label{strpos}G_{g,f}(\nu)=\|\nu\|^2_g+2E_g(\zeta,\nu)\geqslant
\|\nu\|^2_g\geqslant0\text{ \ for all\ }\nu\in\mathcal E^+_g(D).\end{equation}

Consider an exhaustion of $D$ by an increasing sequence $\{K_j\}_{j\in\mathbb N}$ of compact sets, and
write  $A_1^j:=A_1\cap K_j$. As $c_g(A_1^j)<\infty$ for every $j\in\mathbb N$, while $c_g(A_1)=\infty$,
it follows from the subadditivity of $c_g(\cdot)$ on universally measurable sets \cite[Lemma~2.3.5]{F1}
that $c_g(A_1\setminus A_1^j)=\infty$. Hence, for every $j\in\mathbb N$ there is
$\nu_j\in\mathcal E_g^+(A_1\setminus A_1^j,1;D)$ of compact support $S_D^{\nu_j}$ such that
\begin{equation}\label{to0}\|\nu_j\|_g\leqslant 1/j.\end{equation}
Clearly, the $K_j$ can be chosen successively so that $A_1^j\cup S^{\nu_j}_D\subset A_1^{j+1}$.
Any compact set $K\subset D$ is contained in a certain $K_j$ with $j$ large enough, and hence $K$
has points in common with only finitely many $S^{\nu_j}_D$. Therefore $\xi_0$ defined by the relation
\[\xi_0(\varphi):=\sum_{j\in\mathbb N}\,\nu_j(\varphi)\text{ \ for any\ }\varphi\in C_0(D)\]
is a positive Radon measure on $D$ carried by $A_1$. Furthermore, $\xi_0(A_1)=\infty$. For each
$\sigma\in\mathfrak C(A_1)\cup\{\infty\}$ such that $\sigma\geqslant\xi_0$ we thus have
\[\nu_j\in\mathcal E^\sigma_{g,f}(A_1,1;D)\text{ \ for all\ }j\in\mathbb N.\]
Therefore, by (\ref{to0}) and the Cauchy--Schwarz inequality,
\[\lim_{j\to\infty}\,G_{g,f}(\nu_j)=\lim_{j\to\infty}\,\bigl[\|\nu_j\|^2_g+2E_g(\zeta,\nu_j)\bigr]\leqslant
\lim_{j\to\infty}\,\bigl[\|\nu_j\|^2_g+2\|\zeta\|_g\|\nu_j\|_g\bigr]=0.\]
Combined with (\ref{strpos}), this yields
$G^\sigma_{g,f}(A_1,1;D)=0$. Relation (\ref{strpos}) also implies
that $G^\sigma_{g,f}(A_1,1;D)$ can be attained only at zero measure,
for the kernel $g$ is strictly positive definite. Since
$0\not\in\mathcal E^\sigma_{g,f}(A_1,1;D)$, Problem~\ref{3.2} with
the constraint $\sigma$ specified above is unsolvable.\end{proof}

Combining Theorems~\ref{main-th} and~\ref{infcap} leads to the following assertion.

\begin{cor}\label{cor-n-s}Let Case\/~{\rm II} with\/ $\zeta\geqslant0$ take place. Then the\/
$\widetilde{\mathcal H}$-problem as well as Problem\/~{\rm\ref{3.1}}
is solvable for every\/ $\sigma\in\mathfrak C(A_1)\cup\{\infty\}$ if
and only if\/ $c_g(A_1)$ is finite.\end{cor}

\begin{rem}Theorem~\ref{infcap} and Corollary~\ref{cor-n-s} remain valid in Case~I provided that
$f(x)\to0$ as $x\to\omega_D$, where $\omega_D$ is the Alexandroff point of the locally compact space~$D$.
\end{rem}

\section{Description of the supports of the minimizers}\label{desc-sup-sec}

Let $\nu\in\widetilde{\mathcal H}$, resp.\ $\nu\in\mathcal
E^\sigma_{\alpha,f}(\mathbf A,\mathbf 1)$, where $\sigma\in\mathfrak
C(A_1)\cup\{\infty\}$, solve the $\widetilde{\mathcal H}$-problem,
resp.\ Problem~\ref{3.1}. According to Theorem~\ref{lem-equ}, resp.\
Theorem~\ref{th-equiv}, there exists the (unique) solution
$\lambda_{A_1,g}$ to Problem~\ref{3.2}, and moreover
\begin{equation}\label{nunu}\nu^-=(\nu^+)'=\lambda_{A_1,g}'.\end{equation}
The following Theorem~\ref{desc-sup} describes $S_{\mathbb R^n}^{\nu^-}$, while a description of
$S_D^{\nu^+}$ will be provided in Theorems~\ref{cor-desc} and \ref{cor-infty-3} below.

Let $\breve{A}_2$ denote the \textsl{$\kappa_\alpha$-reduced
kernel\/} of  $A_2$ \cite[p.~164]{L}, i.e.\ the set of all $x\in
A_2$ such that $c_\alpha\bigl(B(x,r)\cap A_2\bigr)>0$ for any $r>0$,
where $B(x,r):=\bigl\{y\in\mathbb R^n: \ |y-x|<r\bigr\}$.

For the sake of simplicity of formulation, in the following assertion we assume that in the case $\alpha=2$
the domain $D$ is simply connected.

\begin{thm}\label{desc-sup} $S_{\mathbb R^n}^{\nu^-}$ is described by
\begin{equation}\label{lemma-desc-riesz}
S^{\nu^-}_{\mathbb R^n}=\left\{
\begin{array}{lll} \breve{A}_2 & \text{if} & \alpha<2,\\ \partial D  & \text{if} & \alpha=2.\\
\end{array} \right.
\end{equation}
\end{thm}

\begin{proof}For any $x\in D$ let $K_x$ denote the inverse of $\mathrm{Cl}_{\overline{\mathbb R^n}}A_2$ relative
to $S(x,1)$. Since $K_x$ is compact, there is the (unique)
$\kappa_\alpha$-equil\-ibrium measure $\gamma_x\in\mathcal
E^+_\alpha(K_x;\mathbb R^n)$ on $K_x$ with the properties
$\|\gamma_x\|^2_\alpha=\gamma_x(K_x)=c_\alpha(K_x)$,
\begin{equation}\label{Keq}\kappa_\alpha\gamma_x=1\text{ \ n.e.\ on\ }K_x,\end{equation}
and $\kappa_\alpha\gamma_x\leqslant1$ on $\mathbb R^n$. Note that $\gamma_x\ne0$, for $c_\alpha(K_x)>0$
in consequence of $c_\alpha(A_2)>0$, see \cite[Chapter~IV, Section~5, n$^\circ$\,19]{L}.
We assert that under the stated requirements
\begin{equation}\label{eq-desc}
S^{\gamma_x}_{\mathbb R^n}=\left\{
\begin{array}{lll} \breve{K}_x & \text{if} & \alpha<2,\\ \partial K_x  & \text{if} & \alpha=2.\\ \end{array} \right.
\end{equation}
The latter identity in (\ref{eq-desc}) follows from
\cite[Chapter~II, Section~3, n$^\circ$\,13]{L}. To establish the
former identity,\footnote{We have brought here this proof, since we
did not find a reference for this possibly known assertion.} we
first note that $S^{\gamma_x}_{\mathbb R^n}\subset\breve{K}_x$ by
the $c_\alpha$-absolute continuity of $\gamma_x$. As for the
converse inclusion, assume on the contrary that there is
$x_0\in\breve{K}_x$ such that $x_0\notin S^{\gamma_x}_{\mathbb
R^n}$. Choose $r>0$ with the property  $\overline{B}(x_0,r)\cap
S^{\gamma_x}_{\mathbb R^n}=\varnothing$, where
$\overline{B}(x_0,r):=\bigl\{y\in\mathbb R^n: \ |y-x_0|\leqslant
r\bigr\}$. But $c_\alpha\bigl(B(x_0,r)\cap\breve{K}_x\bigr)>0$,
hence by (\ref{Keq}) there is $y\in B(x_0,r)$ such that
$\kappa_\alpha\gamma_x(y)=1$. The function $\kappa_\alpha\gamma_x$
is $\alpha$-har\-monic on $B(x_0,r)$ \cite[Chapter~I, Section~5,
n$^\circ$\,20]{L}, continuous on $\overline{B}(x_0,r)$, and takes at
$y\in B(x_0,r)$ the maximum value $1$. Applying
\cite[Theorem~1.28]{L} we see that $\kappa_\alpha\gamma_x=1$ holds
$m$-a.e.\ on $\mathbb R^n$, hence everywhere on $(\breve{K}_x)^c$ by
the continuity of $\kappa_\alpha\gamma_x$ on
$\bigl(S^{\gamma_x}_{\mathbb R^n}\bigr)^c$ \
$\bigl[{}\supset(\breve{K}_x)^c\bigr]$, and altogether n.e.\ on
$\mathbb R^n$ by (\ref{Keq}). This means that $\gamma_x$ serves as
the $\alpha$-Riesz equilibrium measure on the whole of $\mathbb
R^n$, which is impossible.

Based on (\ref{nunu}), (\ref{eq-desc}) and the integral
representation (\ref{int-repr}), we then arrive at
(\ref{lemma-desc-riesz}) with the aid of the fact that for every
$x\in D$, $\varepsilon_x'$ is the Kelvin transform of the
equilibrium measure $\gamma_x$, see
\cite[Section~3.3]{FZ-Fin}.\end{proof}

\section{Description of the potentials of the minimizers}\label{desc}

Let a generalized condenser $\mathbf A=(A_1,A_2)$, a constraint
$\sigma\in\mathfrak C(A_1)\cup\{\infty\}$, and an external field $f$
be as indicated in Sections~\ref{sec-aux} and~\ref{sec-statement}.
For any $\nu\in\dot{\mathcal E}_{\alpha,f}^\sigma(\mathbf A,\mathbf
1)$ define $W_{\alpha,f}^\nu:=\kappa_\alpha\nu+f$;
$W_{\alpha,f}^\nu$ is termed the \textsl{$f$-weighted\/
$\alpha$-Riesz potential\/} of $\nu$. Since $\kappa_\alpha\nu$ is
finite n.e.\ on $\mathbb R^n$ by the general convention
(\ref{1.3.10}), $W_{\alpha,f}^\nu$ is well defined n.e.\ on $\mathbb
R^n$ and finite n.e.\ on $A_1^\circ$ (see (\ref{suff})).
Furthermore, $W_{\alpha,f}^\nu$ is $|\mu|$-measurable for every
$\mu\in\mathfrak M(\mathbb R^n)$.

The purpose of this section is to establish  sufficient and/or
necessary conditions for the solvability of the $\widetilde{\mathcal
H}$-problem as well as Problem~\ref{3.1} in terms of variational
inequalities for the $f$-weighted $\alpha$-Riesz potential.

\subsection{Variational inequalities in the constrained minimum $\alpha$-Riesz weak energy
problems}\label{sec-var1-Riesz} Consider first the
$\widetilde{\mathcal H}$-problem as well as Problem~\ref{3.1} in the
constrained setting ($\sigma\ne\infty$). Omitting now the
requirement $G_{g,f}^\sigma(A_1,1;D)<\infty$, for a given
$\xi\in\mathfrak C(A_1)$ we assume instead that $E_g(\xi|_K)<\infty$
for every compact $K\subset A_1^\circ$. Also assume that
\begin{equation}\label{as1}\xi(A_1\setminus A_1^\circ)=0,\end{equation}
which yields $\xi(A_1^\circ)>1$. In view of Lemma~\ref{suff-fin},
$\mathcal E_{g,f}^\xi(A_1,1;D)$ is then nonempty, and hence the
auxiliary Problem~\ref{3.2} makes sense. According to
Theorem~\ref{lem-equ}, resp.\ Theorem~\ref{th-equiv},  so does the
$\widetilde{\mathcal H}$-problem, resp.\ Problem~\ref{3.1}.

\begin{thm}\label{desc-th1}Let\/ $f$ be lower bounded, and let\/ $\nu\in\widetilde{\mathcal H}$,
resp.\ $\nu\in\mathcal E^\xi_{\alpha,f}(\mathbf A,\mathbf 1)$, be
given. Then\/ $\nu$ solves the\/ $\widetilde{\mathcal H}$-problem,
resp.\ Problem\/~{\rm\ref{3.1}}, if and only it satisfies the
following three relations
\begin{align}
\label{sol-bal}\nu^-&=(\nu^+)',\\
\label{desc1'}W^\nu_{\alpha,f}&\geqslant w\quad(\xi-\nu^+)\mbox{-a.e.},\\
W^\nu_{\alpha,f}&\leqslant
w\quad\nu^+\mbox{-a.e.}\label{desc2'}\end{align} with some\/
$w\in\mathbb R$.\footnote{The lower boundedness of $f$ holds
automatically whenever Case~I takes place. Furthermore, in Case~I
(\ref{desc2'}) can be rewritten equivalently in the following
apparently stronger form: $W^\nu_{\alpha,f}\leqslant w$ on
$S^{\nu^+}_D$. Also note that when speaking of the
$\widetilde{\mathcal H}$-problem, (\ref{sol-bal}) holds
automatically, see~(\ref{def-H}).}
\end{thm}

\begin{proof} We begin by establishing Theorem~\ref{desc-th2} below, related to the (auxiliary)
Problem~\ref{3.2}. When investigating Problem~\ref{3.2} we shall
need the following assertion.

\begin{lem} [{\rm see \cite[Lemma~4.3]{DFHSZ}}]\label{lequiv} $\lambda\in\mathcal E_{g,f}^\xi(A_1,1;D)$
solves Problem\/~{\rm\ref{3.2}} if and only if
\[\bigl\langle W_{g,f}^\lambda,\mu-\lambda\bigr\rangle\geqslant0\mbox{ \ for all \ }
\mu\in\mathcal E_{g,f}^\xi(A_1,1;D),\]
where it is denoted\/ $W_{g,f}^\lambda:=g\lambda+f|_D$.
\end{lem}

\begin{thm}\label{desc-th2}Let\/ $f$ be lower bounded. A measure\/
$\lambda\in\mathcal E^\xi_{g,f}(A_1,1;D)$ is the\/
{\rm(}unique\/{\rm)} solution to Problem\/~{\rm\ref{3.2}} if and
only if there is\/ $w\in\mathbb R$ such that
\begin{align}\label{desc1}W^\lambda_{g,f}&\geqslant w\quad(\xi-\lambda)\mbox{-a.e.},\\
W^\lambda_{g,f}&\leqslant
w\quad\lambda\mbox{-a.e.}\label{desc2}\end{align}\end{thm}

\begin{proof} We permanently use the fact that both $\xi$ and $\lambda$ are $c_g$-absolutely continuous,
for they are of finite $\alpha$-Green energy if restricted to any compact $K\subset A^\circ_1$. For any
$c\in\mathbb R$ write
\[A_1^+(c):=\{x\in A_1:\ W_{g,f}^\lambda(x)>c\}\quad\text{and}\quad
A_1^-(c):=\{x\in A_1:\ W_{g,f}^\lambda(x)<c\}.\]

Suppose first that $\lambda$ solves Problem~\ref{3.2}. Inequality
(\ref{desc1}) is valid with $w=L$, where
\[L:=\sup\,\bigl\{q\in\mathbb R: \ W_{g,f}^\lambda\geqslant q\quad(\xi-\lambda)\mbox{-a.e.}\bigr\}.\]
In turn, (\ref{desc1}) with $w=L$ implies that $L<\infty$ because
$W_{g,f}^\lambda<\infty$ holds n.e.\ on $A_1^\circ$ and hence
$(\xi-\lambda)$-a.e.\ on $A_1^\circ$, while
$(\xi-\lambda)(A_1^\circ)>0$ by (\ref{as1}). Also note that
$L>-\infty$, for $W_{g,f}^\lambda$ is lower bounded on $A_1$ by
assumption.

We next proceed by establishing (\ref{desc2}) with $w=L$. Assume, on
the contrary, that this fails, i.e.\ $\lambda(A_1^+(L))>0$. Since
$W_{g,f}^\lambda$ is $\lambda$-measurable, one can choose
$c_1\in(L,\infty)$ so that $\lambda(A_1^+(c_1))>0$. At the same
time, as $c_1>L$, (\ref{desc1}) with $w=L$  yields
$(\xi-\lambda)(A_1^-(c_1))>0$. Therefore, there exist compact sets
$K_1\subset A_1^+(c_1)$ and $K_2\subset A_1^-(c_1)$ such that
\[0<\lambda(K_1)<(\xi-\lambda)(K_2).\]
Write $\tau:=(\xi-\lambda)|_{K_2}$; then $E_g(\tau)<\infty$. Since
$\bigl\langle W_{g,f}^\lambda,\tau\bigr\rangle\leqslant c_1\tau(K_2)<\infty$, we thus get
$\langle f,\tau\rangle<\infty$. Define
$\theta:=\lambda-\lambda|_{K_1}+b\tau$, where $b:=\lambda(K_1)/\tau(K_2)\in(0,1)$ by the last display.
Straightforward verification then shows that $\theta(A_1)=1$ and $\theta\leqslant\xi$, and hence
$\theta\in\mathcal E^\xi_{g,f}(A_1,1;D)$. On the other hand,
\begin{align*}
\langle W_{g,f}^\lambda,\theta-\lambda\rangle&=\langle
W_{g,f}^\lambda-c_1,\theta-\lambda\rangle\\&{}=-\langle
W_{g,f}^\lambda-c_1,\lambda|_{K_1}\rangle+b\langle
W_{g,f}^\lambda-c_1,\tau\rangle<0,\end{align*} which is impossible
by Lemma~\ref{lequiv} applied to $\lambda$ and $\mu=\theta$. This
establishes~(\ref{desc2}).

Conversely, let (\ref{desc1}) and (\ref{desc2}) both hold with some
$w\in\mathbb R$. Then $\lambda(A_1^+(w))=0$ and
$(\xi-\lambda)(A_1^-(w))=0$. For any $\nu\in\mathcal
E^\xi_{g,f}(A_1,1;D)$ we therefore obtain
\begin{align*}\langle
W_{g,f}^\lambda,\nu-\lambda\rangle&=\langle
W_{g,f}^\lambda-w,\nu-\lambda\rangle\\
&{}=\langle W_{g,f}^\lambda-w,\nu|_{A_1^+(w)}\rangle+\langle
W_{g,f}^\lambda-w,(\nu-\xi)|_{A_1^-(w)}\rangle\geqslant0.\end{align*}
Application of Lemma~\ref{lequiv} shows that, indeed, $\lambda$ is
the solution to Problem~\ref{3.2}.
\end{proof}

To complete the proof of Theorem~\ref{desc-th1}, fix
$\nu\in\widetilde{\mathcal H}$, resp.\ $\nu\in\mathcal
E^\xi_{\alpha,f}(\mathbf A,\mathbf 1)$. Note that
\begin{equation}\label{inclusion}\nu^+\in\mathcal E^\xi_{g,f}(A_1,1;D),\end{equation}
which follows from (\ref{def-H}) if $\nu\in\widetilde{\mathcal H}$,
and otherwise it is obvious by $E_g(\nu^+)\leqslant
E_\alpha(\nu^+)$.

Assume first that the given $\nu$ is the (unique) solution to the
$\widetilde{\mathcal H}$-problem, resp.\ Problem~\ref{3.1}.
According to Theorem~\ref{lem-equ}, resp.\ Theorem~\ref{th-equiv},
then (\ref{sol-bal}) holds with $\nu^+=\lambda$, where $\lambda$ is
the (unique) solution to Problem~\ref{3.2}. Therefore by
Lemma~\ref{l-hatg}
\[W^\nu_{\alpha,f}=g\lambda+f|_D=W_{g,f}^\lambda\text{ \ n.e.\ on\ }D.\]
Combined with (\ref{desc1}) and (\ref{desc2}), this leads to
(\ref{desc1'}) and (\ref{desc2'}) (with the same $w$ as in
(\ref{desc1}) and (\ref{desc2})). Here we have used the
$c_\alpha$-absolute continuity of $\lambda$ and $\xi$, see
footnote~\ref{foot-ga}.

Conversely, let for the given $\nu$ all the relations
(\ref{sol-bal}), (\ref{desc1'}), and (\ref{desc2'}) hold true. By
Lemma~\ref{l-hatg}, then $W^\nu_{\alpha,f}=W_{g,f}^{\nu^+}$ n.e.\ on
$D$, where $\nu^+$ belongs to $\mathcal E^\xi_{g,f}(A_1,1;D)$
according to relation (\ref{inclusion}). In view of the
$c_\alpha$-absolute continuity of $\nu^+$ and $\xi$, we thus see
from (\ref{desc1'}) and (\ref{desc2'}) that $W_{g,f}^{\nu^+}$
satisfies (\ref{desc1}) and (\ref{desc2}) (with the same $w$ as in
(\ref{desc1'}) and (\ref{desc2'})), which by Theorem~\ref{desc-th2}
implies that $\nu^+$ is the solution $\lambda$ to Problem~\ref{3.2}.
Substituting $\nu^+=\lambda$ into (\ref{sol-bal}) and then applying
Theorem~\ref{lem-equ}, resp.\ Theorem~\ref{th-equiv}, we infer that
$\nu$ is the solution to the $\widetilde{\mathcal H}$-problem,
resp.\ Problem~\ref{3.1}, thus completing the proof of
Theorem~\ref{desc-th1}.\end{proof}

\begin{thm}\label{cor-desc}Let\/ $f=0$, and assume that\/ $\nu\in\widetilde{\mathcal H}$,
resp.\ $\nu\in\mathcal E^\xi_{\alpha,f}(\mathbf A,\mathbf 1)$,
solves the\/ $\widetilde{\mathcal H}$-problem, resp.\
Problem\/~{\rm\ref{3.1}}. Then, and only then, this\/ $\nu$
satisfies\/ {\rm(\ref{sol-bal})} as well as the following two
relations
\begin{align}\label{desc1''}\kappa_\alpha\nu&=w\quad(\xi-\nu^+)\mbox{-a.e.},\\
\kappa_\alpha\nu&\leqslant w\quad{on\ }\mathbb R^n\label{desc2''}\end{align}
with some\/ $w\in(0,\infty)$. If moreover\/ $\alpha<2$ and\/ $m(D^c)>0$, it is also necessary that
\begin{equation}\label{s1}
S_D^{\nu^+}=S_D^{\xi}.
\end{equation}
\end{thm}

\begin{proof} As seen from Theorem~\ref{desc-th1}, the former part of the theorem will be established
once we have shown that the number $w$ from (\ref{desc1'}) and
(\ref{desc2'}) is now ${}>0$, while the quoted relations can be
rewritten as (\ref{desc1''}) and (\ref{desc2''}). We obtain from
(\ref{sol-bal}) in view of Lemma~\ref{l-hatg}
\begin{equation}\label{repr-eqqq}W^\nu_{\alpha,f}=\kappa_\alpha\nu=g\nu^+\text{ \ on\ }D.\end{equation}
Substituting this into (\ref{desc2'}) gives $w\in(0,\infty)$, while
(\ref{desc2'}) itself now takes the form
\[\kappa_\alpha\nu^+\leqslant w+\kappa_\alpha\nu^-\quad\nu^+\mbox{-a.e.}\]
Consider an exhaustion of $D$ by an increasing sequence $\{K_j\}_{j\in\mathbb N}$ of compact sets,
and denote $\nu^+_j:=\nu^+|_{K_j}$. The last display then remains valid with $\nu^+$ replaced by
$\nu^+_j$, i.e.
\[\kappa_\alpha\nu^+_j\leqslant w+\kappa_\alpha\nu^-\quad\nu^+_j\mbox{-a.e.}\]
Since $\nu^+_j\in\mathcal E_\alpha^+(K_j;\mathbb R^n)$ and $w>0$,
the former being clear from $E_g(\nu^+_j)\leqslant
E_g(\nu^+)<\infty$ and Lemma~\ref{eq-r-g}, application of
\cite[Theorems~1.27, 1.29]{L} shows that the preceding display holds
in fact everywhere on $\mathbb R^n$. As
$\kappa_\alpha\nu^+_j\uparrow\kappa_\alpha\nu^+$ pointwise on
$\mathbb R^n$, letting here $j\to\infty$ results in (\ref{desc2''}).
Combining (\ref{desc2''}) and (\ref{desc1'})
establishes~(\ref{desc1''}).

Assuming now that the hypotheses of the latter part of the theorem
be fulfilled, we
 proceed by establishing (\ref{s1}), which according
to Theorem~\ref{lem-equ}, resp.\ Theorem~\ref{th-equiv}, can be
rewritten in the form  $S_D^{\lambda}=S_D^{\xi}$, where $\lambda$ is
the solution (which exists) to Problem~\ref{3.2}. On the contrary,
let there be $x_0\in S_D^{\xi}$ such that $x_0\not\in
S^{\lambda}_D$. Then one can choose $r>0$ so that
$\overline{B}(x_0,r)\subset D$ and $\overline{B}(x_0,r)\cap
S^{\lambda}_D=\varnothing$. It follows that
$(\xi-\lambda)(B(x_0,r)\cap S_D^{\xi})>0$. By (\ref{desc1''}), there
is therefore $x_1\in B(x_0,r)\cap S_D^{\xi}$ possessing the property
\begin{equation}\label{bempty}\kappa_\alpha\lambda(x_1)=w+\kappa_\alpha\lambda'(x_1).\end{equation}
As $\kappa_\alpha\lambda$ is $\alpha$-harmonic on $B(x_0,r)$ and
continuous on $\overline{B}(x_0,r)$, while $w+\kappa_\alpha\lambda'$
is $\alpha$-sup\-er\-harmonic on $\mathbb R^n$, we see from
(\ref{desc2''}) and (\ref{bempty}) with the aid of
\cite[Theorem~1.28]{L} that
\[\kappa_\alpha\lambda=w+\kappa_\alpha\lambda'\quad m\mbox{-a.e.\ on\ }\mathbb R^n.\]
This implies $w=0$, for $\kappa_\alpha\lambda=\kappa_\alpha\lambda'$ holds n.e.\ on $D^c$,
and hence $m$-a.e.\ on $D^c$. A
contradiction.\end{proof}

\begin{rem}Let $f=0$, and let $\nu$ solve the $\widetilde{\mathcal H}$-problem,
resp.\ Problem~\ref{3.1}, with a \textsl{bounded\/} constraint
$\xi$. Combining (\ref{desc1''}) and (\ref{repr-eqqq}) shows that
$g\nu^+=w$ holds $(\xi-\nu^+)$-a.e. Integrating this equality with
respect to the (bounded positive) measure $\xi-\nu^+$ implies that
the number $w\in(0,\infty)$ from Theorem~\ref{cor-desc} can be
written in the form
\[w=\frac{E_g(\nu^+,\xi-\nu^+)}{(\xi-\nu^+)(A_1)}.\]
Thus $E_g(\nu^+,\xi-\nu^+)<\infty$, though in general $\xi\notin\mathcal E^+_g(D)$.
\end{rem}

\subsection{Variational inequalities in the unconstrained minimum $\alpha$-Riesz weak energy
problems}\label{sec-var2-Riesz} In this section we shall consider the unconstrained case ($\sigma=\infty$).
The results obtained then take a simpler form if compared with those in the constrained case, while
they provide us with much more detailed information about the
potentials and the supports of the minimizers. When $\sigma=\infty$ serves as a superscript,
we shall omit it in the notations.

\begin{thm}\label{desc-infty}Let\/ $\nu\in\widetilde{\mathcal H}$,
resp.\ $\nu\in\mathcal E_{\alpha,f}(\mathbf A,\mathbf 1)$, be given.
Then\/ $\nu$ solves the\/ $\widetilde{\mathcal H}$-problem, resp.\
Problem\/~{\rm\ref{3.1}}, if and only if it satisfies\/
{\rm(\ref{sol-bal})} as well as the following two relations
\begin{align}\label{desc1infty}W^\nu_{\alpha,f}&\geqslant w'\text{ \ n.e.\ on\ }A_1,\\
W^\nu_{\alpha,f}&=w'\text{ \ $\nu^+$-a.e.},\label{desc2infty}\end{align}
where\/ $w'\in\mathbb R$.
\end{thm}

\begin{proof} We first establish the following theorem, related to Problem~\ref{3.2} (with $\sigma=\infty$).

\begin{thm}\label{desc-th2-infty} A measure\/ $\lambda\in\mathcal E^+_{g,f}(A_1,1;D)$ is the\/
{\rm(}unique\/{\rm)} solution to Problem\/~{\rm\ref{3.2}} if and
only if there is\/ $w'\in\mathbb R$ such that\/\footnote{Suppose
Case~I holds. Then (\ref{eqq2}) leads to $W^\lambda_{g,f}\leqslant
w'$ on $S_D^\lambda$, which together with (\ref{eqq2'}) gives
$W^\lambda_{g,f}=w'$ n.e.\ on $S_D^\lambda$. Similarly,
(\ref{desc2infty}) in Theorem~\ref{desc-infty} takes the following
apparently stronger form: $W^\nu_{\alpha,f}=w'$ n.e.\ on
$S_D^{\nu^+}$. Also note that the number $w'$ from
Theorems~\ref{desc-infty} and~\ref{desc-th2-infty} is then~${}>0$.}
\begin{align}\label{eqq2'}W^\lambda_{g,f}&\geqslant w'\text{ \ n.e.\ on  }A_1,\\
\label{eqq2}W^\lambda_{g,f}&=w'\text{ \
$\lambda$-a.e.}\end{align}\end{thm}

\begin{proof}This theorem is in fact a very particular case of \cite[Theorems~7.1, 7.3]{ZPot2}
(see also Theorem~1 and Proposition~1 in the earlier paper \cite{Z5a}).\end{proof}

Theorem~\ref{desc-infty} can now be obtained from
Theorem~\ref{desc-th2-infty} with the aid of Theorems~\ref{th-equiv}
and~\ref{lem-equ} in the same manner as Theorem~\ref{desc-th1} has
been derived from Theorem~\ref{desc-th2}.\end{proof}

\begin{rem}The number $w'$ from Theorems~\ref{desc-infty} and~\ref{desc-th2-infty} is unique
(whenever it exists) and can be written in the form
\[w'=\bigl\langle W^\nu_{\alpha,f},\nu^+\bigr\rangle=\bigl\langle W^\lambda_{g,f},\lambda\bigr\rangle,\]
$\nu$ and $\lambda$ being as indicated in Theorems~\ref{desc-infty} and~\ref{desc-th2-infty}, respectively.
\end{rem}

\begin{defn}\label{def-m-c}For a generalized condenser $\mathbf A$ in $\mathbb R^n$,
$\theta=\theta_{\mathbf A,\alpha}\in\mathfrak M(\mathbf A)$ is said
to be a \textsl{condenser measure\/} if $\kappa_\alpha\theta$ takes
the value $1$ and $0$ n.e.\ on $A_1$ and $A_2$, respectively, and
\[0\leqslant\kappa_\alpha\theta\leqslant1\text{ \ n.e.\ on\ }\mathbb R^n.\]\end{defn}

As seen from the above definition and \cite[p.~178, Remark]{L}, a
condenser measure is \textsl{unique\/} provided that it is
$c_\alpha$-absolutely continuous.

If $A_1$ and $A_2$ are compact disjoint sets, then the existence of a condenser measure was established
by Kishi~\cite{Ki}, actually even in the general setting of a function kernel on a locally compact
Hausdorff space. See also \cite{D3}, \cite{L}, \cite{Bl}, \cite{Berg} where the existence of condenser
potentials was analyzed in the framework of Dirichlet spaces. An intimate relation between a condenser
measure $\theta_{\mathbf A,\alpha}$, $\mathbf A$ being a generalized condenser in $\mathbb R^n$, and
the solution to the $\ddot{\mathcal E}_\alpha(\mathbf A,\mathbf 1)$-problem (see Remark~\ref{rem-FZ-Pot})
has been established in our preceding study~\cite{FZ-Pot}.

\begin{thm}\label{cor-infty-2}Let\/ $f=0$, and let\/ $\nu\in\widetilde{\mathcal H}$, resp.\
$\nu\in\mathcal E_{\alpha,f}(\mathbf A,\mathbf 1)$, be given. Then the following four assertions are
equivalent:
\begin{itemize}
\item[{\rm (i)}] $\nu$ solves the\/ $\widetilde{\mathcal H}$-problem, resp.\ Problem\/~{\rm\ref{3.1}}.
\item[{\rm (ii)}] $\nu$ satisfies\/~{\rm(\ref{sol-bal})} as well as the following two relations
\begin{align}\label{cor01}\kappa_\alpha\nu&=w'\text{ \ on \ }A_1\setminus I_{A_1,\alpha},\\
\label{cor02}\kappa_\alpha\nu&\leqslant w'\text{ \ on\/ \ }\mathbb R^n,
\end{align}
where\/ $w'\in(0,\infty)$.
\item[{\rm (iii)}] There is a\/ {\rm(}unique\/{\rm)} bounded\/ $c_\alpha$-absolutely
continuous condenser
 measure\/~$\theta_{\mathbf A,\alpha}$.
\item[{\rm (iv)}] $c_g(A_1)<\infty$.
\end{itemize}
If any of these\/ {\rm(i)-(iv)} holds, then the number\/ $w'$
appearing in\/ {\rm(ii)} is unique and given by\/\footnote{If the
separation condition (\ref{dist}) holds, then $\dot{E}_\alpha(\nu)$
in (\ref{cor03}) can be replaced by $E_\alpha(\nu)$ (see
Lemma~\ref{l-st}).}
\begin{equation}\label{cor03}w'=\dot{E}_\alpha(\nu)=E_g(\mu_{A_1,g})=1/c_g(A_1).\end{equation}
Furthermore, $\nu$ and\/ $\theta_{\mathbf A,\alpha}$ are related to one another by the formula
\[\theta_{\mathbf A,\alpha}=c_g(A_1)\nu=\gamma_{A_1,g}-\gamma_{A_1,g}'.\]
Here\/ $\mu_{A_1,g}$, resp.\/ $\gamma_{A_1,g}$, is the\/ $g$-capacitary, resp.\ $g$-equilibrium,
measure on\/ $A_1$.
\end{thm}

\begin{proof}Let the assumptions of the theorem be fulfilled, and let (i) hold. According to
Theorem~\ref{cor-un-un}, (i) is equivalent to (iv), and moreover
\begin{equation}\label{cor04'}\nu=\mu_{A_1,g}-\mu_{A_1,g}'=
(\gamma_{A_1,g}-\gamma_{A_1,g}')/c_g(A_1).\end{equation}
Since $f=0$, we thus have by Lemma~\ref{l-hatg}
\begin{equation}\label{xxx}W^\nu_{\alpha,f}=\kappa_\alpha\nu=g\mu_{A_1,g}=
g\gamma_{A_1,g}/c_g(A_1)\text{ \ on\ } D.\end{equation} Combined
with (\ref{geq-2}) for $F=A_1$, this shows that (\ref{cor01}) holds
with $w'=E_g(\mu_{A_1,g})=1/c_g(A_1)\in(0,\infty)$. But according to
Theorem~\ref{thm},
$E_g(\mu_{A_1,g})=\dot{E}_\alpha(\mu_{A_1,g}-\mu_{A_1,g}')=\dot{E}_\alpha(\nu)$,
which together with the preceding relation
establishes~(\ref{cor03}).

Since $\nu$ is $c_\alpha$-absolutely continuous, we see from
(\ref{cor01}) that $\kappa_\alpha\nu^+=w'+\kappa_\alpha\nu^-$ holds
$\nu^+$-a.e. Applying now to $\nu^+$ the same arguments as in the
first paragraph of the proof of Theorem~\ref{cor-desc}, we therefore
arrive at (\ref{cor02}), thus completing the proof of the
implication (i)$\Rightarrow$(ii). The converse implication follows
directly from Theorem~\ref{desc-infty}.

Assuming now again that (i) holds, we next prove that
$\theta:=c_g(A_1)\nu$ is a condenser measure. Combining
$\kappa_\alpha\theta=c_g(A_1)\kappa_\alpha\nu$ with
(\ref{cor01})--(\ref{cor03}) shows that $\kappa_\alpha\theta$ equals
$1$ n.e.\ on $A_1$ and ${}\leqslant1$ on $\mathbb R^n$. But this
$\theta$ can be written as $\gamma_{A_1,g}-\gamma_{A_1,g}'$, see
(\ref{cor04'}), and hence $\kappa_\alpha\theta$ equals $0$ n.e.\ on
$A_2$ by (\ref{bal-eq}). Noting that
$\kappa_\alpha\theta=g\gamma_{A_1,g}>0$ on $D$, see (\ref{xxx}), we
conclude that this $\theta$ is indeed a condenser measure, which in
addition is bounded and $c_\alpha$-absolutely continuous. It has
thus been proven that (i) implies (iii) with $\theta_{\mathbf
A,\alpha}:=\theta$.

To complete the proof, it is enough to show that (iii) implies (iv).
By Definition~\ref{def-m-c},  $\kappa_\alpha\theta_{\mathbf A,\alpha}=0$ n.e.\ on $A_2$,
which in view of the stated $c_\alpha$-absolute continuity of $\theta_{\mathbf A,\alpha}$
implies that $\theta_{\mathbf A,\alpha}^-$ is the $\alpha$-Riesz swept measure of
$\theta_{\mathbf A,\alpha}^+$ onto $A_2=D^c$ (see Section~\ref{sec-bala}). Applying
Lemma~\ref{l-hatg} to the (bounded, hence extendible) measure $\theta_{\mathbf A,\alpha}^+$, we thus get
\[g\theta_{\mathbf A,\alpha}^+=\kappa_\alpha\theta_{\mathbf A,\alpha}=1\text{ \ n.e.\ on\ }A_1.\]
Integrating this equality with respect to the ($c_\alpha$-absolutely continuous, bounded)
measure $\theta_{\mathbf A,\alpha}^+$ implies that $\theta_{\mathbf A,\alpha}^+\in\mathcal E^+_g(A_1;D)$.
Applying \cite[Lemma~3.2.2]{F1} with $\kappa=g$, we therefore see from the last display that
$c_g(A_1)<\infty$, which is~(iv).
\end{proof}

In the following assertion we require that in the case $\alpha<2$, $m(D^c)>0$. For the sake of simplicity
of formulation, we also assume that in the case $\alpha=2$, $D\setminus\breve{A}_1$ is simply connected,
where $\breve{A}_1$ denotes the $\kappa_\alpha$-red\-uc\-ed kernel of $A_1$ (see Section~\ref{desc-sup-sec}).

\begin{thm}\label{cor-infty-3}Let\/ $f=0$, and let\/ $\nu\in\widetilde{\mathcal H}$, resp.\
$\nu\in\mathcal E_{\alpha,f}(\mathbf A,\mathbf 1)$, solve the\/
$\widetilde{\mathcal H}$-problem, resp.\ Problem\/~{\rm\ref{3.1}}.
In addition to\/ {\rm(\ref{cor01})} and\/ {\rm(\ref{cor02})}, then
\begin{equation}\label{cor04}\kappa_\alpha\nu<w'\text{ \ on\ }D\setminus\breve{A}_1,\end{equation}
$w'$ being defined by\/ {\rm(\ref{cor03})}, and also
\begin{equation}\label{s1infty}
 S^{\nu^+}_{D}=\left\{
\begin{array}{lll} \breve{A}_1 & \text{if} & \alpha<2,\\
\partial_D\breve{A}_1 &  \text{if} & \alpha=2.\\ \end{array} \right.
\end{equation}
\end{thm}

\begin{proof}Under the stated hypotheses, $\nu^+=\mu$, where $\mu:=\mu_{A_1,g}$ is the
$g$-capacitary measure on $A_1$, and $g\mu=\kappa_\alpha\nu$ on $D$. Assuming first $\alpha<2$,
we begin by showing that
\begin{equation}\label{lesssup}g\mu<w'\text{ \ on\ } D\setminus S_D^\mu.\end{equation}
Suppose on the contrary that this fails for some $x_0\in D\setminus
S^\mu_D$. By (\ref{cor02}), then $g\mu(x_0)=w'$, or equivalently
\begin{equation}\label{rrr}\kappa_\alpha\mu(x_0)=w'+\kappa_\alpha\mu'(x_0).
\end{equation}
Choose $\varepsilon>0$ so that $\overline{B}(x_0,\varepsilon)\subset
D\setminus S^\mu_D$. Since $\kappa_\alpha\mu$ is $\alpha$-har\-monic
on $B(x_0,\varepsilon)$ and continuous on
$\overline{B}(x_0,\varepsilon)$, while $w'+\kappa_\alpha\mu'$ is
$\alpha$-super\-harmonic on $\mathbb R^n$, we conclude from
(\ref{cor02}) and (\ref{rrr}) with the aid of \cite[Theorem~1.28]{L}
that $\kappa_\alpha\mu=w'+\kappa_\alpha\mu'$  $m$-a.e.\ on $\mathbb
R^n$. As $\kappa_\alpha\mu=\kappa_\alpha\mu'$ holds n.e.\ on $D^c$,
hence $m$-a.e.\ on $D^c$, we thus get $w'=0$. A contradiction.

We next proceed by proving the former identity in (\ref{s1infty}).
Let, on the contrary, there exist $x_1\in\breve{A}_1$ such that
$x_1\not\in S^\mu_D$, and let $V\subset D\setminus S^\mu_D$ be an
open neighborhood of $x_1$. By (\ref{lesssup}), then $g\mu<w'$ on
$V$. On the other hand, since $V\cap A_1$ has nonzero capacity,
$g\mu(x_2)=w'$ for some $x_2\in V$ by (\ref{cor01}). The
contradiction obtained shows that, indeed, $S^\mu_D=\breve{A}_1$.
Substituting this into (\ref{lesssup}) establishes (\ref{cor04}) for
$\alpha<2$.

In the rest of the proof, let $\alpha=2$. To verify (\ref{cor04}),
assume on the contrary that it fails for some $x_3$ in the domain
$D_0:=D\setminus\breve{A}_1$. By (\ref{cor02}), then $g\mu(x_3)=w'$,
which by the maximum principle applied to the harmonic function
$g\mu$ on $D_0$ yields $g\mu=w'$ on $D_0$. Together with
(\ref{cor01}) this shows that $g\mu=w'$ n.e.\ on $D$, so that
$\mu/w'$ serves as the $g$-equilibrium measure on the whole of $D$,
which is impossible.

By \cite[Theorem~1.13]{L}, we see from (\ref{cor01}) that the
restriction of $\mu$ to $\breve{A}_1\setminus\partial_D\breve{A}_1$
equals $0$, and so $S^\mu_D\subset\partial_D\breve{A}_1$. Thus, if
we prove the converse inclusion, then the latter identity in
(\ref{s1infty}) follows. On the contrary, assume there is a point
$y\in\partial_D\breve{A}_1$ such that $y\notin S^\mu_D$; then one
can choose a neighborhood $V_1\subset D$ of $y$ with the property
$V_1\cap S^\mu_D=\varnothing$. As $c_\alpha(V_1\cap A_1)>0$,
(\ref{cor01}) implies that $g\mu(y_1)=w'$ for some $y_1\in V_1$.
Taking (\ref{cor02}) into account and applying the maximum principle
to the harmonic function $g\mu$ on $V_1$, we thus get $g\mu=w'$ on
$V_1$. Since $V_1\cap D_0\ne\varnothing$, this
contradicts~(\ref{cor04}).\end{proof}

\section{Comments}

\begin{rem}Based on Theorem~\ref{cor-infty-3} and its proof, we are led to the following assertion,
providing a description of the $\alpha$-Green equilibrium measure $\gamma$ on $F$.
As usual, we denote by $\breve{F}$ the $\kappa_\alpha$-red\-uc\-ed kernel of $F$
(see Section~\ref{desc-sup-sec}).

\begin{thm}\label{eq-m-desc} Let\/ $F$ be a relatively closed subset of\/ $D$ with $c_g(F)<\infty$.
If\/ $\alpha<2$, assume additionally that\/ $m(D^c)>0$, while in the case\/ $\alpha=2$
let\/ $D\setminus\breve{F}$ be simply connected. Then the support of
the\/ $\alpha$-Green equilibrium measure\/ $\gamma$ on\/ $F$ is given by
\[
 S^{\gamma}_{D}=\left\{
\begin{array}{lll} \breve{F} &  \text{if} & \alpha<2,\\
\partial_D\breve{F} &  \text{if}  & \alpha=2.\\ \end{array} \right.
\]
\end{thm}
\end{rem}

\begin{rem}\label{dr-dif}As seen from the results obtained, the generalized condensers such that the
unconstrained minimum  weak $\alpha$-Riesz energy problems are
solvable differ drastically from those for which the solvability
occurs in the constrained setting. Indeed, if the constraint
$\xi\in\mathfrak C(A_1)$ is \textsl{bounded\/}, then  the
$\widetilde{\mathcal H}$-problem as well as Problem~\ref{3.1} is
solvable in either Case~I or Case~II even if
$c_\alpha(A_2\cap\mathrm{Cl}_{\mathbb R^n}A_1)>0$ (actually, even if
$A_1=D$; see Theorem~\ref{main-th}). However, if $f=0$ and
$\sigma=\infty$ ({\it no\/} external field and {\it no\/} active
constraint), then the $\widetilde{\mathcal H}$-problem as well as
Problem~\ref{3.1} reduces to problem (\ref{cap-def}) with $Q=A_1$
and $\kappa=g$, or equivalently to the problem on the existence of
the $\alpha$-Green equilibrium measure $\gamma_{A_1}$, while the
solvability of the latter necessarily implies that
$c_\alpha(A_2\cap\mathrm{Cl}_{\mathbb R^n}A_1)=0$ (see
Corollary~\ref{lusin}).
\end{rem}

\section{An example of a Green equilibrium measure with infinite Newtonian energy}\label{appen}

Let $n=3$, $\alpha=2$, and let $D=\bigl\{(x_1,x_2,x_3)\in\mathbb R^3:\ x_1>0\bigr\}$. We construct
a relatively closed $2$-regular subset $F$ of $D$ with $0<c_g(F)<\infty$ such that its (classical)
Green equilibrium measure $\gamma=\gamma_F$ (which exists, see Theorem~\ref{th-equi}) has infinite
Newtonian energy. The present example is a strengthening of the example in \cite[Appendix]{DFHSZ2},
quoted in the Introduction, because the measure in question now is an equilibrium measure.

This example shows that the results of the present paper, related to minimum weak
$\alpha$-Riesz problems over subclasses of $\mathfrak M(\mathbf A)$, in general fail
if we replace the weak energy by the standard $\alpha$-Riesz energy. This is also the case for
\cite[Theorem~6.1]{FZ-Pot}, quoted in Remark~\ref{rem-FZ-Pot} above. This justifies the need for
the concept of weak $\alpha$-Riesz energy when dealing with condenser problems.

\begin{exmp}\label{counterex2} Let $D$ be a domain in $\mathbb R^3$, specified above.
The boundary $\partial D$ is then the plane $\{x_1=0\}$. For $r>0$
write $K_r:=\bigl\{(0,x_2,x_3)\in\mathbb R^3: \ x_2^2+x_3^2\leqslant
r^2\bigr\}$; then $K_r$ is the closed disc in the plane $\partial D$
of radius $r$ centered at $(0,0,0)$. Write briefly $K:=K_1$. For
$\varepsilon\in\mathbb R$ let $K_r^\varepsilon$ denote the
translation of $K_r$ by the vector $(\varepsilon,0,0)$. Thus
$K_r^\varepsilon\subset D$ when $\varepsilon>0$. For
$\varepsilon,s>0$ denote by $K_{r,s}^\varepsilon\subset D$ the
translation of $K_r$ by $(\varepsilon,s,0)$.

Since $0<c_2(K)<\infty$ (in fact $c_2(K)=2/\pi^2$, see e.g.\
\cite[Chapter~II, Section~3, n$^\circ$\,14]{L}), there exists the
(unique) $\kappa_2$-capacitary measure $\mu$ on $K$ (see
Remarks~\ref{ex-perf} and~\ref{remark}). Its Newtonian potential
$\kappa_2\mu$ on $\mathbb R^3$ is then constant everywhere on the
disc $K$, e.g.\ by the Wiener criterion, and equals there the
Newtonian energy $E_2(\mu)=1/c_2(K)$. By the continuity principle
\cite[Theorem~1.7]{L}, $\kappa_2\mu$ is (finitely) continuous on
$\mathbb R^3$, and even uniformly since $\kappa_2\mu(x)\to0$
uniformly as $|x|\to\infty$, the support $S_{\mathbb R^3}^{\mu}$
being compact (actually, $S_{\mathbb R^3}^{\mu}=K$).

For any positive Radon measure $\nu$ on $\mathbb R^3$ we denote by
$\hat{\nu}$ the image of $\nu$ under the reflection
$(x_1,x_2,x_3)\mapsto(-x_1,x_2,x_3)$ with respect to $\partial D$.
Similarly, for any $y=(y_1,y_2,y_3)\in\mathbb R^3$ write
$\hat{y}=(-y_1,y_2,y_3)$. The \textsl{$2$-Green kernel\/} $g^2_D=g$
on the half-space $D$ is then given by
\[g(x,y)=\kappa_2(x,y)-\kappa_2(x,\hat{y})=\kappa_2\varepsilon_y(x)-\kappa_2\widehat{\varepsilon_y}(x),\]
see e.g.\ \cite[Theorem 4.1.6]{AG}. This yields the homogeneity relation
\begin{equation}\label{g-r} g(rx,ry)=r^{-1}g(x,y)\text{ \ for\ }x,y\in D, \ r>0.\end{equation}

Writing for brevity $\psi(\varepsilon):=c_g(K_1^\varepsilon)$, we have for $\varepsilon,r\in(0,\infty)$
\begin{equation}\label{phi}c_g(K_r^\varepsilon)=rc_g(K^{\varepsilon/r}_1)=
r\psi\Bigl(\frac\varepsilon r\Bigr).\end{equation} In fact, if
$\gamma_\varepsilon$ denotes the $g$-equilibrium measure on
$K_1^{\varepsilon/r}$, then the image $r\circ\gamma_\varepsilon$ of
$\gamma_\varepsilon$ under the homothety $x\mapsto rx$, $x\in D$
(which preserves the total mass) is carried by $K_r^\varepsilon$ and
has constant $g$-potential $1/r$ in view of (\ref{g-r}). Therefore,
$r(r\circ\gamma_\varepsilon)$ is the $g$-equilibrium measure on
$K_r^\varepsilon$, and hence $c_g(K_r^\varepsilon)$ equals the total
mass
$r\gamma_\varepsilon(K_1^{\varepsilon/r})=rc_g(K_1^{\varepsilon/r})$.
In particular,
\begin{equation}\label{phi-1}\psi(\delta)=c_g(K_1^\delta)=
\delta c_g(K_{1/\delta}^1)\text{ \ for any\ }\delta\in(0,\infty). \end{equation}

\begin{lem}\label{lemma} The function\/ $\psi$ is\/ {\rm(}finitely\/{\rm)} continuous. Moreover,
\[\lim_{\delta\to0}\,\psi(\delta)=\infty.\]\end{lem}

\begin{proof} The value $\psi(\delta)$ is finite because $K_1^\delta$ is compact in $D$. In view of
(\ref{phi}) and (\ref{phi-1}), for continuity of $\psi$ it suffices
to establish continuity of $r\mapsto c_g(K_r^1)$ for
$r\in(0,\infty)$. This mapping is continuous from the right at any
$r_0$ because $K_{r_0}^1=\bigcap_{r>r_0}K_r^1$ with $K_r^1$ compact
and decreasing for $r\downarrow r_0$ (see \cite[Lemma~4.2.1]{F1}).
For any $r\in(0,\infty)$ denote by $\partial_r$ the boundary of
$K_r^1$ relative to the plane $\{x_1=1\}$, i.e.\
$\partial_r:=\partial_{\{x_1=1\}}K_r^1$. For continuity of
$c_g(K_r^1)$ from the left, note that
$K_{r_0}^1\setminus\partial_{r_0}=\bigcup_{r<r_0}\,K_r^1$ (for
example, through an increasing sequence of $r$), hence
$c_g(K_{r_0}^1)=\sup_{r<r_0}c_g(K_r^1)$ by \cite[Theorem~4.2]{F1}
because the Newtonian capacity of $\partial_{r_0}$, imbedded  into
$\mathbb R^3$, equals $0$, and hence so does $c_g(\partial_{r_0})$
(footnote~\ref{foot-ga}).

Denoting by $\mu^\delta$ the image of $\mu$ under the translation
$(\delta,0,0)$, $\mu$ being the $\kappa_2$-cap\-acitary measure on
$K$, we get for any $\delta>0$
\begin{align}
\label{to00}
E_g(\mu^\delta)
  &=\int g\mu^\delta\,d\mu^\delta=\int\Bigl(\,\kappa_2\mu^\delta-\kappa_2\widehat{\mu^\delta}\Bigr)\,
  d\mu^\delta\\
{}&=\int\kappa_2\mu\,d\mu-\int\kappa_2\mu(-2\delta,x_2,x_3)\,d\mu(x_1,x_2,x_3).\notag\end{align}
Note that $\kappa_2\mu(-2\delta,x_2,x_3)\to\kappa_2\mu(0,x_2,x_3)$
uniformly with respect to $(0,x_2,x_3)\in K$ as $\delta\to0$, which
is seen from the uniform continuity of $\kappa_2\mu$ on $\mathbb
R^3$ established above. In view of (\ref{to00}), we therefore get
\[\psi(\delta)=c_g(K_1^\delta)\geqslant 1/E_g(\mu^\delta)\to\infty\text{ \ as \ } \delta\to0,\]
$\mu^\delta$ being the $\kappa_2$-capacitary measure on $K_1^\delta$.\end{proof}

For the construction of the desired relatively compact subset $F$ of
$D$ with finite $c_g(F)$ we consider sequences of numbers
$\varepsilon_j>0$, $r_j=j^{-3}$, and $s_j:=aj\uparrow\infty$ for
some constant $a\geqslant4$. The set $F$ will be of the form
\begin{equation}\label{F-Fj}F:=\bigcup_jF_j,\quad F_j:=K^{\varepsilon_j}_{r_j,s_j}.\end{equation}
Such $F$ is indeed relatively closed in $D$, for if a sequence $\{x_k\}_{k\in\mathbb N}\subset F$
converges to $x\in D$, then all the $x_k$ lie in a suitable (compact) finite union of sets $F_j$
because $s_j\to\infty$.

Define $b:=\inf_\delta\,\psi(\delta)<\infty$. In view of Lemma
\ref{lemma} we choose for every integer $j>b$ a number $\delta_j$
such that $\psi(\delta_j)=j$. Next choose
$\varepsilon_j=r_j\delta_j=j^{-3}\delta_j$ for $j>b$ and note that
with summation over $j>b$ we have by (\ref{phi}), (\ref{phi-1}),
and~(\ref{F-Fj})
\[\sum_{j}\,c_g(F_j)=\sum_{j}\,c_g(K_{r_j}^{\varepsilon_j})=
\sum_{j}\,
r_jc_g(K_1^{\varepsilon_j/r_j})=\sum_{j}\,j^{-3}\psi(\delta_j)=\sum_{j}\,j^{-2}<\infty.\]
In view of the countable subadditivity of $c_g(\cdot)$ on
universally measurable sets, see \cite[Lemma~2.3.5]{F1}, it follows
that $c_g(F)<\infty$, and so there exists the (unique)
$g$-equilibrium measure $\gamma=\gamma_F$ on $F$. This positive
measure $\gamma$ has constant $g$-potential $1$ everywhere on $F$
because every point of $F$ (that is, of some $F_j$) is $2$-regular,
as noted earlier.

Denoting $\gamma_j:=\gamma|_{F_j}$ (where
$F_j=K_{r_j,s_j}^{\varepsilon_j}$), we next show that
\begin{equation}\label{lambda-j}\frac12\leqslant g\gamma_j\leqslant1\text{ \ on\ }F_j.\end{equation}
Fix $a:=\max\{\gamma(F),4\}$ for our choice of the sequence $s_j=aj$. For any $x\in F_j$,
\begin{equation}\label{small-g} \sum_{k\ne j}\,g\gamma_k(x)\leqslant
\sum_{k\ne j}\,\kappa_2\gamma_k(x)=\sum_{k\ne
j}\,\int_{F_k}\frac1{|x-y|}\,d\gamma_k(y) \leqslant\sum_{k\ne
j}\,\frac{\gamma(F_k)}{a-2}\leqslant\frac{\gamma(F)}{a/2}\leqslant\frac12\end{equation}
after using that $r_j,r_k\leqslant1$ and
$|x-y|\geqslant|s_j-s_k|-(r_j+r_k)\geqslant a-2\geqslant
a/2\geqslant2$ for $x\in F_j$, $y\in F_k$. It follows by subtraction
that $g\gamma_j\geqslant1/2$ on $F_j$ and of course
$g\gamma_j\leqslant g\gamma\leqslant1$ on $D$, thus
establishing~(\ref{lambda-j}).

Since $g(2\gamma_j)\geqslant1$ on $F_j$ by (\ref{lambda-j}) and
\[c_g(F_j)=\inf\,\bigl\{\nu(F_j): \ \nu\in\mathfrak M^+(F_j), \
g\nu\geqslant1\text{\ n.e.\ on\ }F_j\,\bigr\}\]
(see e.g.\ \cite[Theorem~5.5.5(ii)]{AG} or \cite[p.~243]{Doob}), we thus get
\begin{equation}\label{chain}c_g(F_j)\leqslant2\gamma(F_j).\end{equation}

As $\gamma_j$ is carried by $F_j=K^{\varepsilon_j}_{r_j,s_j}$ with
diameter $2r_j$, we have for the Newtonian energy $E_2(\gamma)$
after summation over $j>b$:
\[E_2(\gamma)\geqslant\sum_{j}\,E_2(\gamma_j)\geqslant\sum_{j}\,\frac{\gamma(F_j)^2}{2r_j}\geqslant
\sum_{j}\,\frac{c_g(F_j)^2}{8r_j}=
\sum_{j}\,\frac{(r_j\psi(\delta_j))^2}{8r_j}=\sum_{j}\,\frac1{8j}=\infty,\]
where the third inequality holds by (\ref{chain}), the first
equality by (\ref{phi}) and (\ref{phi-1}), and the second equality
by our choices $r_j=j^{-3}$ and $\psi(\delta_j)=j$.
\end{exmp}


\begin{thebibliography}{99}

\setlength{\parskip}{1.2ex plus 0.5ex minus 0.2ex}

\bibitem{AG}Armitage, D.H., Gardiner, S.J.: Classical Potential Theory. Springer, Berlin (2001)

\bibitem{Bagby} Bagby, T.: The modulus of a plane condenser. J.\ Math.\ Mech. {\bf 17}, 315--329 (1967)

\bibitem{Berg} Berg, C.: On the existence of condenser potentials. Nagoya Math. J. {\bf 70}, 157--165 (1978)

\bibitem{Bl} Bliedtner, J.: Dirichlet Forms on Regular Functional Spaces. Lecture Notes in Math., vol.~226. Springer,
Berlin (1971)

\bibitem{BH} Bliedtner, J., Hansen, W.: Potential Theory: An Analytic and Probabilistic Approach to Balayage. Springer, Berlin (1986)

\bibitem{B1}
Bourbaki, N.: Elements of Mathematics, General Topology, Chapters~1--4.
Springer, Berlin (1989)

\bibitem{B2} Bourbaki, N.: Elements of Mathematics, Integration, Chapters 1--6.
Springer, Berlin (2004)

\bibitem{Brelo2} Brelot, M.: On Topologies and Boundaries in Potential Theory. Lecture Notes in Math., vol.~175. Springer,
Berlin (1971)

\bibitem{Ca} Cartan, H.: Th\'eorie du potentiel newtonien: \'energie, capacit\'e, suites de potentiels. Bull.\ Soc.\ Math.\ France {\bf 73}, 74--106 (1945)

\bibitem{Ca2} Cartan, H.: Th\'eorie g\'en\'erale du balayage en potentiel newtonien. Ann.\ Univ.\ Grenoble {\bf 22}, 221--280 (1946)

\bibitem{D1} Deny, J.: Les potentiels d'\'energie finie. Acta Math.\ {\bf 82}, 107--183 (1950)

\bibitem{D2} Deny, J.: Sur la d\'efinition de l'\'energie en th\'eorie du potentiel. Ann. Inst. Fourier {\bf 2}, 83--99 (1950)

\bibitem{D3} Deny, J.: Sur les espaces de Dirichlet. S\'{e}m. Th\'{e}orie du potentiel, no.~5 (1957)

\bibitem{Doob} Doob, J.L.: Classical Potential Theory and Its Probabilistic Counterpart. Springer, Berlin (1984)

\bibitem{DFHSZ} Dragnev, P.D., Fuglede, B., Hardin, D.P., Saff, E.B., Zorii, N.: Minimum Riesz energy problems for a condenser with touching plates. Potential Anal.\ {\bf 44}, 543--577 (2016)

\bibitem{DFHSZ2} Dragnev, P.D., Fuglede, B., Hardin, D.P., Saff, E.B., Zorii, N.: Condensers with touching plates and constrained minimum Riesz and Green energy problems. Constr.\ Approx., to appear.  ArXiv:1711.05484 (2017)

\bibitem{DFHSZ1} Dragnev, P.D., Fuglede, B., Hardin, D.P., Saff, E.B., Zorii, N.: Constrained minimum Riesz energy problems for a condenser with intersecting plates. J.~Anal.\ Math., to appear. ArXiv:1710.01950v2 (2018)

\bibitem{F1} Fuglede, B.: On the theory of potentials in locally compact spaces. Acta Math. {\bf 103}, 139--215 (1960)

\bibitem{FZ-Fin} Fuglede, B., Zorii, N.: Green kernels associated with Riesz kernels. Ann.\ Acad.\ Sci.\ Fenn.\ Math.\ {\bf 43},  121--145 (2018)

\bibitem{FZ-Pot} Fuglede, B., Zorii, N.: An alternative concept of Riesz energy of measures with application to generalized condensers, Potential Anal., https://doi.org/10.1007/s11118-018-9709-3

\bibitem{HWZ}
Harbrecht, H., Wendland, W.L., Zorii, N.: Riesz minimal
energy problems on $C^{k-1,k}$-manifolds. Math.\ Nachr.\ {\bf 287}, 48--69 (2014)

\bibitem{Ki} Kishi, M.: Sur l'existence des mesures des condensateurs. Nagoya Math. J. {\bf 30}, 1--7 (1967)

\bibitem{L} Landkof, N.S.: Foundations of Modern Potential Theory. Springer, Berlin (1972)

\bibitem{OWZ}
Of, G., Wendland, W.L., Zorii, N.: On the numerical solution of minimal energy problems. Complex Var.\
Elliptic Equ.\ {\bf 55}, 991--1012 (2010)

\bibitem{O}
Ohtsuka, M.: On potentials in locally compact spaces,
J.~Sci.\ Hiroshima Univ.\ Ser.~A-1 {\bf 25}, 135--352 (1961)

\bibitem{R} Riesz, M.: Int\'egrales de Riemann--Liouville et potentiels. Acta Szeged {\bf 9}, 1--42 (1938)

\bibitem{S} Schwartz, L.: Th\'eorie des Distributions, vol.~I. Hermann, Paris (1951)

\bibitem{Z0} Zorii, N.: An extremal problem of the minimum of energy for space
condensers. Ukrain.\ Math.~J. {\bf 38}, 365--369 (1986)

\bibitem{ZR} Zorii, N.: A problem of minimum energy for space condensers and Riesz kernels.  Ukrain.\ Math.~J. {\bf 41}, 29--36 (1989)

\bibitem{Z5a} Zorii, N.: Equilibrium potentials with external fields. Ukrain.\ Math.~J. {\bf 55}, 1423--1444 (2003)

\bibitem{ZPot1} Zorii, N.: Interior capacities of condensers in locally compact spaces. Potential Anal. {\bf 35}, 103--143 (2011)

\bibitem{Z-MN}
Zorii, N.: Constrained energy problems with external fields for vector measures.
 Math.\ Nachr.\ {\bf 285}, 1144--1165 (2012)

\bibitem{ZPot2} Zorii, N.: Equilibrium problems for infinite dimensional vector potentials with
external fields. Potential Anal.\ {\bf 38}, 397--432 (2013)

\end{thebibliography}
\end{document}